\documentclass[11pt]{article}
\usepackage[margin=1in]{geometry}
\usepackage{amsmath,amssymb,mathtools,amsthm}
\usepackage{booktabs}
\usepackage{microtype}
\usepackage{graphicx}
\usepackage{xcolor}
\usepackage{algorithm}
\usepackage{algorithmic}
\usepackage{flafter}
\usepackage{needspace}
\usepackage{tikz}
\usepackage{pgfplots}
\usepackage[hidelinks]{hyperref}
\usepackage[capitalize,nameinlink]{cleveref}
\pgfplotsset{compat=1.18}
\usetikzlibrary{arrows.meta,positioning}
\usepgfplotslibrary{groupplots}
\hypersetup{
  hidelinks,
  pdftitle={Parameter-Robust Subspace Correction with Multiple Semidefinite Penalties},
  pdfauthor={Subhransu S. Bhattacharjee}
}

\newcommand{\R}{\mathbb{R}}
\newcommand{\N}{\mathcal{N}}
\newcommand{\W}{\mathcal{W}}
\newcommand{\norm}[1]{\lVert #1\rVert}
\newcommand{\range}{\operatorname{range}}
\newcommand{\email}[1]{\href{mailto:#1}{#1}}

\newtheorem{theorem}{Theorem}
\newtheorem{lemma}[theorem]{Lemma}
\newtheorem{corollary}[theorem]{Corollary}
\newtheorem{proposition}[theorem]{Proposition}

\theoremstyle{definition}
\newtheorem{example}[theorem]{Example}
\theoremstyle{remark}
\newtheorem{remark}[theorem]{Remark}
\crefname{example}{example}{examples}
\Crefname{example}{Example}{Examples}
\crefname{remark}{remark}{remarks}
\Crefname{remark}{Remark}{Remarks}
\newenvironment{keywords}{\par\medskip\noindent\textbf{Keywords:}\ }{\par}

\newcommand{\MultipenaltySubsetRows}{%
  left & 210 & 210 & 0 \\%
  right & 210 & 210 & 0 \\%
  joint & 67 & 67 & 0 \\%
}
\newcommand{\MultipenaltyExactRows}{%
  $(1,1)$ & 13.56 \\%
  $(10^{8},1)$ & 72.18 \\%
  $(1,10^{8})$ & 72.18 \\%
  $(10^{8},10^{8})$ & 87.45 \\%
  $(10^{8},10^{2})$ & 82.47 \\%
  $(10^{2},10^{8})$ & 82.47 \\%
}
\newcommand{\MultipenaltyPrimaryRows}{%
  1 & 418 & 7 & 9 & 9 & 10 & 10 & 10 \\%
  2 & 1,602 & 7 & 12 & 12 & 12 & 13 & 13 \\%
  3 & 6,274 & 7 & 12 & 12 & 13 & 13 & 13 \\%
  4 & 24,834 & 7 & 13 & 13 & 13 & 13 & 13 \\%
  5 & 98,818 & 7 & 13 & 13 & 13 & 13 & 13 \\%
}

\newcommand{\MultipenaltyControlRows}{%
  component-matched & 12 & 12 & 13 \\%
  native & $300^\ast$ & $300^\ast$ & $300^\ast$ \\%
  uniform maximum & 14 & 14 & 13 \\%
}

\newcommand{\MultipenaltyBaselineLeftDivergence}{5.82}
\newcommand{\MultipenaltyBaselineRightDivergence}{5.91}
\newcommand{\MultipenaltyLeftStiffLeftDivergence}{2.34\times10^{-7}}
\newcommand{\MultipenaltyLeftStiffRightDivergence}{5.93}
\newcommand{\MultipenaltyRightStiffRightDivergence}{2.40\times10^{-7}}
\newcommand{\MultipenaltyRightStiffLeftDivergence}{5.82}

\newcommand{\ControlSubsetRows}{%
  left & 210 & 170 & 40 & $1.4336$ \\%
  right & 210 & 170 & 40 & $1.4336$ \\%
  joint & 67 & 0 & 67 & $2.1228$ \\%
}
\newcommand{\ControlRateRows}{%
  $10^{2}$ & 0.44674 & $5.16\times10^{-2}$ & $670.66$ \\%
  $10^{4}$ & 0.47081 & $5.54\times10^{-4}$ & $6.37\times10^{4}$ \\%
  $10^{6}$ & 0.47107 & $5.55\times10^{-6}$ & $6.37\times10^{6}$ \\%
  $10^{8}$ & 0.47107 & $4.05\times10^{-7}$ & $6.37\times10^{8}$ \\%
}
\newcommand{\ControlFreeDofs}{354}
\newcommand{\ControlPatches}{57}

\newcommand{\ControlJointPredicted}{0.47107}
\newcommand{\ThreeSubsetRows}{%
  $\{1\}$ & 602 & 602 & 0 & $<10^{-14}$ \\%
  $\{2\}$ & 602 & 601 & 1 & $0.6942$ \\%
  $\{3\}$ & 602 & 602 & 0 & $<10^{-14}$ \\%
  $\{1,2\}$ & 386 & 386 & 0 & $<10^{-14}$ \\%
  $\{1,3\}$ & 386 & 386 & 0 & $<10^{-14}$ \\%
  $\{2,3\}$ & 386 & 386 & 0 & $<10^{-14}$ \\%
  $\{1,2,3\}$ & 171 & 171 & 0 & $<10^{-14}$ \\%
}
\newcommand{\ThreeOrthantRows}{%
  $(1, 1, 1)$ & $29.6$ \\%
  $(10^{8}, 1, 1)$ & $163.48$ \\%
  $(1, 10^{8}, 1)$ & $2.08\times10^{8}$ \\%
  $(1, 1, 10^{8})$ & $163.48$ \\%
  $(10^{8}, 10^{8}, 1)$ & $462.44$ \\%
  $(10^{8}, 1, 10^{8})$ & $188.5$ \\%
  $(1, 10^{8}, 10^{8})$ & $462.44$ \\%
  $(10^{8}, 10^{8}, 10^{8})$ & $420.58$ \\%
  $(10^{8}, 10^{2}, 10^{2})$ & $297.97$ \\%
  $(10^{2}, 10^{8}, 10^{2})$ & $8.94\times10^{6}$ \\%
  $(10^{2}, 10^{2}, 10^{8})$ & $297.97$ \\%
}

\newcommand{\ThreeFreeDofs}{818}
\newcommand{\ThreePatches}{49}
\newcommand{\ThreeBaseN}{6}
\newcommand{\ThreeMiddleGamma}{0.6942}
\newcommand{\ThreeMiddlePredicted}{1.44047}
\newcommand{\ThreeMiddleTailError}{2.12\times10^{-6}}
\newcommand{\ThreeLatticeMembers}{16}
\newcommand{\ThreeCommutator}{2.421}

\pgfplotstableread[col sep=comma]{
epsilon,point,vertex,cone
1.0,82.62473941288671,47.40309775668702,7.472340939330371
0.31622776601683794,107.36311206074046,61.59168071636041,10.052746046977079
0.1,282.4297714887628,148.24131714269447,13.227609774613029
0.03162277660168379,836.0862324976029,370.5189546846035,16.021708573415786
0.01,2586.920659338358,788.071211307921,17.457398927458275
0.0031622776601683794,8123.551775601872,1251.4549226832557,18.129672074707152
0.001,25631.91878647178,1541.7844212552561,18.355969754527845
0.00031622776601683794,80998.23732712198,1664.3865218575518,18.428992931349818
0.0001,256081.90978739905,1707.3763779532449,18.45223413319784
3.1622776601683795e-05,809745.0959127952,1721.4427666416325,18.459598671047875
1e-05,2560581.8208532664,1725.9398962633954,18.461929048049758
3.162277660168379e-06,8097213.68170599,1727.366966321596,18.462666128884948
1e-06,25605580.85364731,1727.8187422318463,18.462899229901193
3.162277660168379e-07,80971899.73297812,1727.9616570058554,18.46297294682444
1e-07,256055573.6224454,1728.0068540635502,18.462996272639895
3.162277660168379e-08,809718804.6925306,1728.0211397182659,18.46300366992673
1e-08,2560554919.3286424,1728.025645642293,18.463006015729803
3.1622776601683795e-09,8097189732.281808,1728.0271999796623,18.46300732787091
1e-09,25605520093.129086,1728.0276022933504,18.4630067308221
3.1622776601683795e-10,80971621693.35808,1728.0298218180199,18.463010230016604
1e-10,256052226709.8521,1728.0306382142157,18.463014664682227
}\PenaltySweepData
\pgfplotstableread[col sep=comma]{
n,base_condition,stiff_condition
4,4.262846724215376,9.455410872312049
8,5.968030550604486,16.49051805038173
12,7.472340939330371,18.463006015729803
16,8.270519194639245,19.16592147594529
20,8.732837865681958,19.562093869663467
24,9.039413876451903,19.784574648240056
28,9.273394838087363,19.922904606531347
32,9.469119793925575,20.017472918514414
36,9.636199196008896,20.086531889691717
40,9.777231600544923,20.139322014268117
}\ScalabilityData
\pgfplotstableread[col sep=comma]{
n,bound
4,9.464509717124695
6,15.106347770563056
8,16.687906451831847
10,18.155187171086787
12,18.6966714836777
14,19.17027978127068
16,19.459203336813527
}\OneParameterBoundData

\title{Parameter-Robust Subspace Correction \\
with Multiple Semidefinite Penalties}
\author{Subhransu S. Bhattacharjee\thanks{%
\mbox{The Australian National University, Canberra, ACT 2601, Australia}\\
\mbox{Corresponding author: \email{Subhransu.Bhattacharjee@anu.edu.au}.}}}
\date{}

\begin{document}
\maketitle

\begin{abstract}
Independently weighted semidefinite penalties arise in augmented-Lagrangian and
constrained formulations.  This paper characterizes when an exact additive
subspace-correction preconditioner remains uniformly effective over all
nonnegative penalty weights on a fixed finite-dimensional space.  Robustness
holds precisely when the correction spaces decompose every joint kernel
generated by a nonempty subset of penalties.  If one condition fails, a
computable constant determines the exact first-order decay of the smallest
preconditioned eigenvalue along the associated parameter ray, and the condition
number grows linearly; none of the subset conditions can be discarded in
general.  Filtered decompositions provide computable sufficient bounds on
parameter-ordering cones, while distributive kernel lattices permit a single
common splitting.  Exact-additive computations confirm the characterization
and predicted rates.  Separate Scott--Vogelius experiments produce stable
multilevel iteration counts over the tested weights and mesh levels.  The
analysis does not establish mesh-uniformity.
\end{abstract}


\section{Introduction}\label{sec:introduction}

Augmented-Lagrangian preconditioners for incompressible flow replace a difficult
pressure Schur-complement approximation with one based on a scaled pressure
mass matrix.  This modification introduces a large semidefinite grad-div term
in the velocity block~\cite{benzi2006augmented,farrell2019augmented}.  As its
weight increases, the velocity energy separates sharply between the
divergence-free subspace and its complement, and unmodified multigrid
components deteriorate.  Robust schemes therefore use relaxation and
intergrid transfer that preserve the penalty kernel~\cite{schoberl1999parameter,
farrell2021reynolds,farrell2022elasticity,shih2023augmented}.

That principle concerns one penalty and one parameter.  Independent weights
appear whenever a constraint is enforced with different strength in different
subdomains, several constraints are augmented at once, or a multiphysics block
inherits parameters from unrelated material
data~\cite{laakmann2022mhd,hong2019conservative,cai2026thermoporoelasticity}.
The distinction is substantive because the weights may diverge separately.  If
$\tau_1$ alone grows, the penalty vanishes on $\ker K_1$; if $\tau_2$ alone
grows, it vanishes on $\ker K_2$; simultaneous growth leaves their intersection
unpenalized.  Different parameter paths therefore expose different limiting
subspaces, and correction spaces suitable for one path need not suffice for
another.

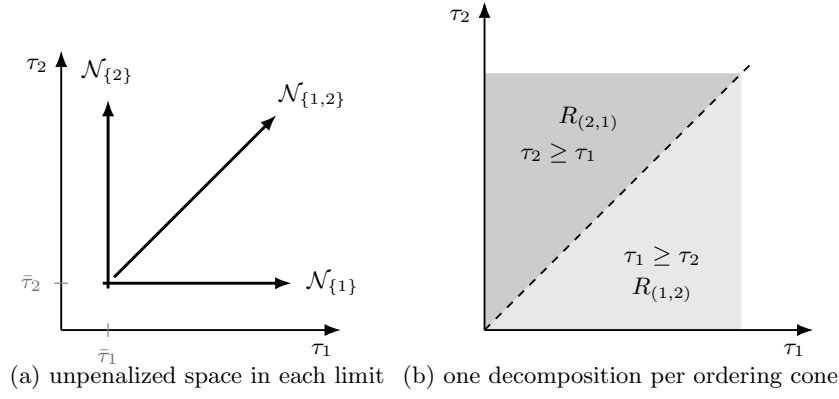
\begin{figure}[tbhp]
\centering
\begin{tikzpicture}[>=Latex,font=\footnotesize,
  ax/.style={->,line width=.7pt},
  esc/.style={-{Latex[length=2.2mm]},line width=1.1pt}]

\begin{scope}
  \draw[ax] (0,0) -- (3.70,0) node[below left=1pt and -1pt]{$\tau_1$};
  \draw[ax] (0,0) -- (0,3.70) node[below left=-1pt and 1pt]{$\tau_2$};
  \draw[black!45,line width=.5pt] (-.09,.62) -- (.09,.62);
  \draw[black!45,line width=.5pt] (.62,-.09) -- (.62,.09);
  \node[anchor=east,black!55,font=\scriptsize] at (-.12,.62) {$\bar\tau_2$};
  \node[anchor=north,black!55,font=\scriptsize] at (.62,-.12) {$\bar\tau_1$};
  \draw[esc] (.55,.62) -- (3.05,.62);
  \draw[esc] (.62,.55) -- (.62,3.05);
  \draw[esc] (.70,.70) -- (2.85,2.85);
  \node[anchor=west]  at (3.12,.62)  {$\N_{\{1\}}$};
  \node[anchor=south] at (.62,3.12)  {$\N_{\{2\}}$};
  \node[anchor=south west] at (2.75,2.80) {$\N_{\{1,2\}}$};
  \node at (1.8,-.62) {(a) unpenalized space in each limit};
\end{scope}

\begin{scope}[xshift=5.6cm]
  \fill[black!9] (0,0) -- (3.4,0) -- (3.4,3.4) -- cycle;
  \fill[black!20] (0,0) -- (3.4,3.4) -- (0,3.4) -- cycle;
  \draw[ax] (0,0) -- (4.35,0) node[below left=1pt and -1pt]{$\tau_1$};
  \draw[ax] (0,0) -- (0,4.35) node[below left=-1pt and 1pt]{$\tau_2$};
  \draw[dashed,line width=.7pt] (0,0) -- (3.55,3.55);
  \node at (2.35,.98) {$\tau_1\ge\tau_2$};
  \node at (2.35,.53)  {$R_{(1,2)}$};
  \node at (.98,2.35) {$\tau_2\ge\tau_1$};
  \node at (1.38,2.80) {$R_{(2,1)}$};
  \node at (1.8,-.62) {(b) one decomposition per ordering cone};
\end{scope}
\end{tikzpicture}
\caption{Two penalties.  (a) Which subspace survives depends on which weights
diverge.  By \cref{thm:characterization} robustness on the quadrant is
equivalent to decomposing all three; the intersection alone does not suffice
(\cref{ex:counterexample}), nor do the singletons (\cref{thm:irredundant}).
(b) The
sufficiency proof splits the quadrant along $\tau_1=\tau_2$ and builds a
separate $R_\pi$ on each cone, where the weights are ordered and the operator
telescopes into nested semidefinite faces.}
\label{fig:orthant}
\end{figure}

We determine the joint kernels that a fixed additive correction must
decompose.  For a nonempty subset $J$ of the penalty indices, let
$\N_J=\bigcap_{j\in J}\ker K_j$ and let $\N_J^{\mathrm{loc}}$ denote the span
assembled from the intersections of the correction spaces with $\N_J$.

\paragraph{Contributions} The main results are as follows.
\begin{enumerate}
 \item On a fixed discretization, the exact additive preconditioner
 is uniformly conditioned over the whole parameter orthant if and only if
 $\N_J^{\mathrm{loc}}=\N_J$ for every nonempty $J\subseteq[m]$
 (\cref{thm:characterization}).  Neither the full intersection alone nor the
 collection of singleton kernels is sufficient.
 \item If the equality fails for $J$, the generalized eigenvalue $\gamma_J$
 determines the exact first-order decay of the smallest preconditioned
 eigenvalue along the associated ray.  The corresponding condition number has
 linear-order growth, and the dichotomy excludes intermediate powers of the
 ray parameter (\cref{thm:necessary,thm:rate}).
 \item The conditions are irredundant in the worst case: for each nonempty
 $J\subseteq[m]$, there is a problem for which precisely that condition fails
 (\cref{thm:irredundant}).
 \item Sufficiency is proved one ordering cone at a time, which makes the
 bound computable from generalized eigenvalues (\cref{cor:cone}).  One
 decomposition serves every cone when the joint nullspaces generate a
 distributive lattice (\cref{thm:distributive}).  This property is automatic
 for two penalties but can fail for three.
\end{enumerate}

The last result distinguishes the characterization from the classical
componentwise sufficient condition, which assumes one decomposition stable in
every penalty energy.  In \cref{ex:ordering-dependent}, three penalties in
$\R^2$ have distinct kernel lines and admit no such common decomposition, yet
the preconditioned condition number is at most $2$ throughout the orthant.  The
cone construction in \cref{fig:orthant}(b) remains applicable because, within
each ordering cone, the operator telescopes into nested semidefinite faces
respected by one filtered decomposition.

\Cref{sec:study} evaluates the theory in three settings.  Dense staggered-grid
calculations permit direct computation of the spectral constants.  Exact
additive correction on Scott--Vogelius spaces evaluates the kernel criterion on a
finite-element discretization, including deliberately deficient correction
spaces.  Finally, an inexact W-cycle provides empirical evidence outside the
hypotheses of the exact-additive theory.  These evidence modes are reported
separately.

\section{Related work}\label{sec:related}

The stable-decomposition and overlap estimates used in additive subspace
correction are classical~\cite{xu1992subspace,griebel1995abstract,
toselli2005domain}.  The exact identity in \cref{prop:identity} is their
finite-dimensional spectral form.  Hackbusch considered additive Schwarz
methods for cones generated by positive-semidefinite component matrices and
showed that estimates valid component by component give a uniform bound for
every nonnegative combination~\cite{hackbusch1992frequency}.  The common
stable decomposition in \cref{thm:uniform} is therefore a classical sufficient
condition.  We restate it to fix the constants used below and to compare it
with the weaker characterization in \cref{sec:necessary}.

For one dominant semidefinite form, Sch\"oberl related robust multigrid to a
local decomposition of the limiting kernel~\cite{schoberl1999parameter}.
Subsequent work treated singular and nearly singular systems, sharp
subspace-correction estimates, and semicoercive
optimization~\cite{lee2007nearly,lee2008sharp,wu2014parallel,
lee2025semicoercive}, while exact-sequence and auxiliary-space arguments gave
robust solvers for weighted $H(\mathrm{div})$ and $H(\mathrm{curl})$ problems
and for parameter-dependent hybridizable discontinuous Galerkin
systems~\cite{arnold2000multigrid,xu1996auxiliary,fu2021uniform}.  Brenner's
analysis of parameter-dependent mechanics and the metric-perturbed framework of
Budi\v{s}a et al.\ likewise use a single relative singular
parameter~\cite{brenner1996multigrid,budisa2024metric}.  These analyses
motivate the kernel condition but treat a single relative singular parameter
rather than independently diverging semidefinite terms.

The flow application belongs to the augmented-Lagrangian literature initiated
for constrained boundary-value problems~\cite{fortin1983augmented}.  Grad-div
augmentation and the associated block factorization were analyzed for Stokes
and Oseen systems in~\cite{olshanskii2004graddiv,benzi2006augmented}; see
\cite{benzi2005saddle} for a review.  Later work produced robust velocity
solvers for high-Reynolds-number flow, Scott--Vogelius elements, strong
viscosity variation, nearly incompressible elasticity, and the Oseen--Frank
model~\cite{farrell2019augmented,farrell2021reynolds,shih2023augmented,
farrell2022elasticity,xia2021oseenfrank}, and multiparameter block
preconditioners arise in magnetohydrodynamics, poroelasticity,
thermo-poroelasticity, and hybridized
discretizations~\cite{laakmann2022mhd,hong2019conservative,piersanti2021congruence,
cai2026thermoporoelasticity,henriquez2025parameter}.  These exploit
problem-specific block norms or factorizations rather than one fixed additive
correction for a cone of semidefinite perturbations.

High-contrast Schwarz methods often obtain robustness by enriching a coarse
space with local generalized eigenvectors~\cite{efendiev2012robust,
spillane2014geneo,aldaas2025robust}.  Generalized eigenvalues play a different
role here: rather than selecting modes with which to enrich a coarse space,
they evaluate a prescribed decomposition and quantify the extent of its failure
on a joint kernel.

Within the scope reviewed above, a necessary-and-sufficient criterion is not
available for one fixed additive correction over a full parameter orthant when
no componentwise-stable reconstruction exists.  \Cref{sec:necessary}
provides such a fixed-dimensional criterion, quantifies failure through the
smallest preconditioned eigenvalue, and proves the irredundance of its subset
conditions.  Finite distributive lattices of subspaces and their compatible
bases are classical~\cite{gratzer2011lattice,harrison2014reflexivity,
loday2012algebraic}; we use this structure to identify when the cone-dependent
decompositions can be replaced by one common decomposition.

\section{The operator family}\label{sec:notation}

This section introduces the operator family and correction spaces and relates
the abstract formulation to augmented-Lagrangian preconditioning.

All spaces and matrices are real and finite dimensional.  For symmetric $H$,
$H\succ0$ and $H\succeq0$ mean positive definite and positive semidefinite,
$H\preceq G$ denotes the L\"owner order, and $\ker H$ and $\range H$ are the
nullspace and range.  For $H\succeq0$, define the energy seminorm
$\norm{v}_H^2=v^\top Hv$, which is a norm when $H\succ0$; $I_X$ is the identity
on $X$.  If $A$ and $P$ are symmetric positive definite (SPD), then
\begin{equation}
 \kappa(P^{-1}A)
 =\frac{\lambda_{\max}(A^{1/2}P^{-1}A^{1/2})}
        {\lambda_{\min}(A^{1/2}P^{-1}A^{1/2})}.
\end{equation}

Throughout, $m$ and $s$ are positive integers, $[m]=\{1,\ldots,m\}$, and an
intersection over the empty index set means the full ambient space.

\paragraph{The operator family}
For each positive discretization parameter $h$, let $V_h=\R^{n_h}$, let
$A_{0,h}\succ0$, and let $K_{j,h}:V_h\to\R^{r_{j,h}}$ for $j\in[m]$.  Define
\begin{equation}
 \begin{aligned}
 A_{j,h}&=K_{j,h}^\top K_{j,h},
    &&j\in[m],\\
 A_h(\boldsymbol\tau)&=A_{0,h}+\sum_{j=1}^m\tau_jA_{j,h},
    &&\boldsymbol\tau\in[0,\infty)^m.
 \end{aligned}
 \label{eq:family}
\end{equation}
Thus $A_h(\boldsymbol\tau)\succ0$ for every admissible parameter vector, and
we omit $h$ for one fixed discretization.

\paragraph{Subspaces and assembly}
Let $V=\sum_{i=1}^s I_iV_i$ with $I_i:V_i\to V$ injective, set
$\W=\bigoplus_{i=1}^sV_i$, and let the assembly map $\mathcal I:\W\to V$ be
$\mathcal I(v_1,\ldots,v_s)=\sum_iI_iv_i$.  For $\ell=0,\ldots,m$, define
$A_{\ell,i}=I_i^\top A_\ell I_i$,
$\mathcal D_\ell=\operatorname{diag}_{i=1}^s(A_{\ell,i})$, and
$A_i(\boldsymbol\tau)=I_i^\top A(\boldsymbol\tau)I_i$.
For $v\in V$, let $\mathfrak D(v):=\{w\in\W:\mathcal Iw=v\}$ be the affine
space of its subspace decompositions.  The exact additive
subspace-correction preconditioner is
\begin{equation}
 P(\boldsymbol\tau)^{-1}
 =\sum_{i=1}^sI_iA_i(\boldsymbol\tau)^{-1}I_i^\top.
 \label{eq:schwarz}
\end{equation}
A decomposition operator is a linear map
$R=(R_1,\ldots,R_s)^\top:V\to\W$ such that $\mathcal IR=I_V$.

\paragraph{Joint kernels}
For each nonempty $J\subseteq[m]$, stack the matrices $K_j$, $j\in J$, to
form $K_J$.  The associated operator, local operator, joint kernel, and
assembled local span are
\begin{equation}
 \begin{aligned}
 A_J&=\sum_{j\in J}A_j,&
 \mathcal D_J&=\sum_{j\in J}\mathcal D_j,\\
 \N_J&=\ker K_J=\ker A_J
      =\bigcap_{j\in J}\ker K_j,&
 \N_J^{\mathrm{loc}}&=\sum_{i=1}^s(I_iV_i\cap\N_J).
 \end{aligned}
 \label{eq:faces}
\end{equation}
For the empty set, set $A_\varnothing=0$, $\mathcal D_\varnothing=0$,
$\N_\varnothing=V$, and $\N_\varnothing^{\mathrm{loc}}=V$.
The constants $C_{\rm sd}$ and $C_{\rm ov}$ are the stable-decomposition and
overlap constants in \cref{prop:identity}.  The energy stability constants
$C_\ell$ and $q_\ell$ are defined in
\crefrange{eq:componentstable}{eq:overlap}.

Inequalities between semidefinite forms are used repeatedly, so we record the
standard criterion.

\begin{lemma}[Semidefinite domination]\label{lem:domination}
Let $H\succeq0$ and $G\succeq0$.  A finite $c$ with $H\preceq c\,G$ exists if
and only if $\ker G\subseteq\ker H$.  When it exists, the least such $c$ is
the largest eigenvalue of the pencil $(Y^\top HY,\,Y^\top GY)$, where the
columns of $Y$ form a basis of $(\ker G)^\perp$; if $G=0$, the least $c$ is
zero~\cite{golub2013matrix}.
\end{lemma}

\begin{proof}
If $H\preceq c\,G$, then $x\in\ker G$ implies $x^\top Hx=0$ and hence
$Hx=0$.  Conversely, suppose $\ker G\subseteq\ker H$ and choose $Y$ with
orthonormal columns spanning $(\ker G)^\perp$.  Both forms vanish on
$\ker G$, while $Y^\top GY$ is positive definite.  The smallest admissible
$c$ is therefore
\begin{equation}
 \max_{y\ne0}\frac{y^\top Y^\top H Yy}{y^\top Y^\top GYy},
\end{equation}
which is the stated generalized eigenvalue.  If $G=0$, the kernel inclusion
forces $H=0$, and the least constant is zero.
\end{proof}

\subsection{Augmented-Lagrangian origin}\label{sec:augmentation}

The family in \cref{eq:family} arises directly from augmented-Lagrangian
preconditioning.  Consider a symmetric saddle-point matrix
\begin{equation}
 \mathcal A=
 \begin{bmatrix}F&B^\top\\ B&0\end{bmatrix},
 \qquad F\succ0,
 \label{eq:saddle}
\end{equation}
where $B$ has full row rank.  Let $M_p$ be an SPD pressure weight.  The standard
augmented-Lagrangian construction~\cite{benzi2006augmented} adds the constraint
with a parameter $\gamma>0$ and replaces $F$ by
\begin{equation}
 F_\gamma=F+\gamma B^\top M_p^{-1}B.
 \label{eq:augmentedblock}
\end{equation}
The augmentation leaves the solution of \cref{eq:saddle} unchanged for
homogeneous constraints.  With $S=BF^{-1}B^\top$ and
$S_\gamma=BF_\gamma^{-1}B^\top$, the Woodbury identity gives
\begin{equation}
 S_\gamma^{-1}=S^{-1}+\gamma M_p^{-1},
 \label{eq:schuridentity}
\end{equation}
so $\gamma^{-1}S_\gamma^{-1}\to M_p^{-1}$ as $\gamma\to\infty$.  At the same
time the augmentation increases the energy of velocity components outside
$\ker B$ while leaving the energy on $\ker B$ unchanged: this is the nearly
singular velocity problem addressed by kernel-preserving relaxation.

For constant viscosity, writing $F=\nu F_0$ and dividing
\cref{eq:augmentedblock} by $\nu$ leaves
$F_0+(\gamma/\nu)B^\top M_p^{-1}B$, so the effective penalty is proportional
to $\nu^{-1}$.  If several constraints $B_j$ are augmented at once, each with
its own SPD weight $M_{p,j}$ and its own parameter, the velocity block has the
form \cref{eq:family}, with $K_j=M_{p,j}^{-1/2}B_j$.  In
\cref{sec:regional-penalties}, restricting the discrete divergence form to two
mesh-aligned subdomains produces two independently weighted semidefinite
penalties.
\section{The classical sufficient condition}\label{sec:upper}

The classical route to parameter robustness asks for a single reconstruction
that is stable in the energy of $A_0$ and of each penalty separately.  That
hypothesis is sufficient, and \cref{sec:necessary} shows it to be
strictly stronger than robustness requires.  We first record the exact identity
underlying every such estimate, in which the spectral endpoints of the
preconditioned operator are the optimal stable-decomposition and overlap
constants~\cite{xu1992subspace,toselli2005domain}.

\begin{proposition}[Exact subspace-correction identity]\label{prop:identity}
Let $A$ be symmetric positive definite, and suppose that the injective maps
$I_i:V_i\to V$ satisfy $V=\sum_iI_iV_i$.  Let
$A_i=I_i^\top A I_i$, $\mathcal D=\operatorname{diag}_i(A_i)$, and
$P^{-1}=\sum_i I_iA_i^{-1}I_i^\top$.  Define
\begin{equation}
 \begin{aligned}
 C_{\rm sd}(A)
   &=\sup_{v\ne0}\frac{\displaystyle\min_{w\in\mathfrak D(v)}
      \norm{w}_{\mathcal D}^2}{\norm v_A^2},\\
 C_{\rm ov}(A)
   &=\sup_{w\ne0}\frac{\norm{\mathcal Iw}_A^2}
      {\norm{w}_{\mathcal D}^2}.
 \end{aligned}
 \label{eq:schwarz-constants}
\end{equation}
Each $A_i$ is SPD, as is $P^{-1}$, and the minimum in
\cref{eq:schwarz-constants} is attained.
The spectral endpoints satisfy
\begin{equation}
 \begin{aligned}
  \lambda_{\min}(P^{-1}A)&=C_{\rm sd}(A)^{-1},&
  \lambda_{\max}(P^{-1}A)&=C_{\rm ov}(A),\\
  \kappa(P^{-1}A)&=C_{\rm sd}(A)C_{\rm ov}(A).&&
 \end{aligned}
 \label{eq:schwarzidentity}
\end{equation}
Moreover, $1\le C_{\rm ov}(A)\le s$.
\end{proposition}

\begin{proof}
Set $T=A^{1/2}\mathcal I\mathcal D^{-1/2}$, so that
$A^{1/2}P^{-1}A^{1/2}=TT^\top$ is similar to $P^{-1}A$.  For $w=(v_i)_i$ and
$z=\mathcal D^{1/2}w$ we have
$\norm{\mathcal Iw}_A^2/\norm{w}_{\mathcal D}^2=\norm{Tz}_2^2/\norm z_2^2$,
so $C_{\rm ov}(A)=\sigma_{\max}(T)^2$.  Because $\mathcal I$ is onto, $T$ has
full row rank; with $T^\dagger$ its Moore--Penrose inverse,
\begin{equation}
 \min_{w\in\mathfrak D(v)}\norm{w}_{\mathcal D}^2
 =\min_{Tz=A^{1/2}v}\norm z_2^2
 =\norm{T^\dagger A^{1/2}v}_2^2 .
\end{equation}
As $A^{1/2}v$ ranges over $V$, this gives
$C_{\rm sd}(A)=\norm{T^\dagger}_2^2=\sigma_{\min}(T)^{-2}$, proving
\cref{eq:schwarzidentity}.  Taking one nonzero component bounds
$C_{\rm ov}$ below, and
$\norm{\sum_iI_iv_i}_A^2\le s\sum_i\norm{I_iv_i}_A^2$ bounds it above.
\end{proof}

\subsection{A common stable decomposition}

Fix a decomposition operator $R=(R_1,\ldots,R_s)^\top:V\to\W$ and assume it is
stable in every component energy, independently of $\boldsymbol\tau$: for
$\ell=0,\ldots,m$,
\begin{equation}
 \sum_i (R_iv)^\top A_{\ell,i}(R_iv)\le C_\ell v^\top A_\ell v
 \quad\text{for every }v\in V.
 \label{eq:componentstable}
\end{equation}
For semidefinite $A_\ell$, \cref{eq:componentstable} implies
$I_iR_iv\in\ker A_\ell$ for every $v\in\ker A_\ell$ and every $i$.  Let
$q_\ell$ be the least constant such that, for every $(v_i)_i\in\W$,
\begin{equation}
 \Big(\sum_iI_iv_i\Big)^\top A_\ell\Big(\sum_iI_iv_i\Big)
 \le q_\ell\sum_iv_i^\top A_{\ell,i}v_i.
 \label{eq:overlap}
\end{equation}
Such a constant always exists and satisfies $q_\ell\le s$: apply the
Cauchy--Schwarz inequality to the vectors $A_\ell^{1/2}I_iv_i$.

\begin{theorem}[Uniform parameter bound]\label{thm:uniform}
Assume \cref{eq:componentstable}.  Set
$C=\max_{0\le\ell\le m}C_\ell$ and $q=\max_{0\le\ell\le m}q_\ell$.  Then
\begin{equation}
 \kappa\,\left(P(\boldsymbol\tau)^{-1}A(\boldsymbol\tau)\right)\le Cq
 \qquad\text{for every }\boldsymbol\tau\in[0,\infty)^m.
 \label{eq:uniformbound}
\end{equation}
For a mesh family, if
$\sup_h\max_\ell C_\ell(h)<\infty$ and
$\sup_h\max_\ell q_\ell(h)<\infty$, then the condition numbers are uniform in
both $h$ and $\boldsymbol\tau$.
\end{theorem}

\begin{proof}
Set $\theta_0=1$ and $\theta_j=\tau_j$.  For the decomposition $v_i=R_iv$,
linearity of the energy gives
\begin{equation}
 \sum_i\norm{v_i}_{A_i(\boldsymbol\tau)}^2
 =\sum_{\ell=0}^m\theta_\ell\sum_i\norm{R_iv}_{A_{\ell,i}}^2
 \le C\sum_{\ell=0}^m\theta_\ell\norm v_{A_\ell}^2
 =C\norm v_{A(\boldsymbol\tau)}^2 ,
\end{equation}
so the optimal stable-decomposition constant is at most $C$, and the same
summation applied to \cref{eq:overlap} bounds the overlap constant by $q$.
\Cref{prop:identity} identifies the product of the two optimal constants with
$\kappa(P(\boldsymbol\tau)^{-1}A(\boldsymbol\tau))$, proving
\cref{eq:uniformbound}.
\end{proof}

\section{Which subspaces must be resolved}\label{sec:necessary}

Fix a nonempty $J\subseteq[m]$.  The space $\N_J$ of \cref{eq:faces} consists
of the directions on which the penalties indexed by $J$ have no effect, however
large their weights, and its local span $\N_J^{\mathrm{loc}}$ consists of those
among them that the prescribed subspaces can assemble without leaving $\N_J$.
The analysis below quantifies the gap between these two spaces.

We hold $J$ fixed, freeze the inactive weights at $\bar\tau_k\ge0$ for
$k\notin J$, and let the remaining ones grow along a common ray:
\begin{equation}
 \begin{aligned}
 \bar A&=A_0+\sum_{k\notin J}\bar\tau_kK_k^\top K_k,\\
 A(t)&=\bar A+t\sum_{j\in J}K_j^\top K_j=\bar A+tA_J .
 \end{aligned}
 \label{eq:active-ray}
\end{equation}
Here $\bar A\succ0$ because $A_0$ is.  Write $A_i(t)=I_i^\top A(t)I_i$,
$P(t)^{-1}=\sum_i I_iA_i(t)^{-1}I_i^\top$, and
$\mathcal D(t)=\operatorname{diag}_iA_i(t)$.

The governing quantity is the least penalty energy required to decompose a
vector on which the penalty vanishes.  For $v\in V$ set
\begin{equation}
 \eta_J(v):=\min_{w\in\mathfrak D(v)}\norm{w}_{\mathcal D_J}^2
 =\min_{(v_i)\in\mathfrak D(v)}
   \sum_i\sum_{j\in J}\norm{K_jI_iv_i}_2^2 .
 \label{eq:kernel-obstruction}
\end{equation}

\begin{lemma}[Penalty-decomposition cost]\label{lem:obstruction}
The minimum in \cref{eq:kernel-obstruction} is attained, and there is a matrix
$G_J\succeq0$ with $\eta_J(v)=v^\top G_Jv$.  A linear map $S_J:V\to\W$ also
attains the minimum, so that $\mathcal IS_J=I_V$ and
$\norm{S_Jv}_{\mathcal D_J}^2=\eta_J(v)$ for every $v$.  Moreover
$\eta_J(v)=0$ if and only if $v\in\N_J^{\mathrm{loc}}$.
\end{lemma}

\begin{proof}
Fix a decomposition operator $R_0$, let the columns of $Z$ span
$\ker\mathcal I$, so that $\mathfrak D(v)=\{R_0v+Zc\}$, and put
$F=\mathcal D_J^{1/2}$.  Then
\begin{equation}
 \eta_J(v)=\min_c\norm{FR_0v+FZc}_2^2=\norm{(I-\Pi)FR_0v}_2^2 ,
\end{equation}
where $\Pi$ projects orthogonally onto $\range(FZ)$; the minimum distance to a
finite-dimensional subspace is attained.  Hence
$G_J=R_0^\top F(I-\Pi)FR_0\succeq0$
represents $\eta_J$, and the least-norm minimizer
$c=-(FZ)^\dagger FR_0v$ is linear in $v$, so
$S_J=R_0-Z(FZ)^\dagger FR_0$ is linear, satisfies $\mathcal IS_J=I_V$, and
attains the minimum.  Finally $\eta_J(v)=0$ exactly when some decomposition
has $K_jI_iv_i=0$ for all $i$ and all $j\in J$, that is, when $v$ is a sum of
vectors $I_iv_i\in\N_J$.
\end{proof}

Define the \emph{rate constant}
\begin{equation}
 \gamma_J=\max\left\{\frac{\eta_J(v)}{\norm v_{\bar A}^2}:
                     v\in\N_J\setminus\{0\}\right\},
 \label{eq:rate-constant}
\end{equation}
with $\gamma_J=0$ when $\N_J=\{0\}$.  If the columns of $U_J$ form a basis of
$\N_J$, then $\gamma_J$ is the largest eigenvalue of the pencil
$(U_J^\top G_JU_J,\,U_J^\top\bar AU_J)$, so it is computable by one
generalized eigenvalue problem.  \Cref{lem:obstruction} gives
\begin{equation}
 \gamma_J>0
 \quad\Longleftrightarrow\quad
 \N_J^{\mathrm{loc}}\ne\N_J .
 \label{eq:rate-positive}
\end{equation}

\begin{theorem}[Necessary kernel condition]\label{thm:necessary}
With $\gamma_J$ as in \cref{eq:rate-constant},
\begin{equation}
 \kappa(P(t)^{-1}A(t))\ge\gamma_Jt \qquad (t\ge0).
 \label{eq:linearbad}
\end{equation}
In particular, if $\N_J^{\mathrm{loc}}\ne\N_J$ then $\gamma_J>0$ and the
condition number grows at least linearly along the ray.  Hence robustness for
all nonnegative parameters requires $\N_J^{\mathrm{loc}}=\N_J$ for every
nonempty $J\subseteq[m]$.
\end{theorem}

\begin{proof}
Let $v\in\N_J$ be nonzero.  Then $v^\top A(t)v=v^\top\bar Av$, while every
$w\in\mathfrak D(v)$ satisfies
$\norm w_{\mathcal D(t)}^2\ge t\norm w_{\mathcal D_J}^2\ge t\eta_J(v)$.
Taking the supremum over such $v$ in \cref{eq:schwarz-constants} gives
$C_{\rm sd}(A(t))\ge t\gamma_J$.  Since $C_{\rm ov}(A(t))\ge1$,
\cref{prop:identity} proves \cref{eq:linearbad}.  The remaining statements
follow from \cref{eq:rate-positive}.
\end{proof}

The converse requires a right inverse that respects a nested sequence of
kernels.  We first record the elementary linear-algebra construction.

\begin{lemma}[Filtered right inverse]\label{lem:filtered}
Let $\mathcal I:\W_0\to\N_0$ be onto, and let
\begin{equation}
 \W_0\supseteq\W_1\supseteq\cdots\supseteq\W_m,
 \qquad
 \N_0\supseteq\N_1\supseteq\cdots\supseteq\N_m
\end{equation}
be flags of finite-dimensional spaces such that
$\mathcal I(\W_k)=\N_k$ for $k=0,\ldots,m$.  Then there is a linear map
$R:\N_0\to\W_0$ such that $\mathcal IR=I_{\N_0}$ and
$R\N_k\subseteq\W_k$ for every $k$.
\end{lemma}

\begin{proof}
Choose $X_k$ such that $\N_k=X_k\oplus\N_{k+1}$ for $k<m$, and set
$X_m=\N_m$.  For each $k$, surjectivity of
$\mathcal I:\W_k\to\N_k$ gives a linear lift $E_k:X_k\to\W_k$ satisfying
$\mathcal IE_k=I_{X_k}$.  Since $\N_0=\bigoplus_{k=0}^mX_k$, the direct sum of
the maps $E_k$ defines a right inverse $R$.  If $v\in\N_k$, only its
components in $X_k,\ldots,X_m$ can be nonzero, and their lifts belong to
$\W_k$ because the flag is nested.
\end{proof}

\begin{theorem}[Characterization by joint-kernel decompositions]
\label{thm:characterization}
Fix the finite-dimensional family in \cref{eq:family}, and suppose that the
subspaces in \cref{eq:schwarz} span $V$.  The exact additive
subspace-correction preconditioner satisfies
\begin{equation}
 \sup_{\boldsymbol\tau\in[0,\infty)^m}
 \kappa\,\left(P(\boldsymbol\tau)^{-1}A(\boldsymbol\tau)\right)<\infty
 \label{eq:orthant-robustness}
\end{equation}
if and only if
\begin{equation}
 \N_J^{\mathrm{loc}}=\N_J
 \qquad\text{for every nonempty }J\subseteq[m].
 \label{eq:all-subset-condition}
\end{equation}
\end{theorem}

\begin{proof}
Necessity follows from \cref{thm:necessary}.  To prove sufficiency, define
$\W_J=\{(v_i)_i\in\W:I_iv_i\in\N_J\text{ for every }i\}$, so that
$\mathcal I(\W_J)=\N_J^{\mathrm{loc}}$.  Fix a permutation
$\pi=(\pi_1,\ldots,\pi_m)$ of $[m]$, set $J_k=\{\pi_1,\ldots,\pi_k\}$ for
$1\le k\le m$, and put $J_0=\varnothing$, $\N_{J_0}=V$, $\W_{J_0}=\W$.
Condition \cref{eq:all-subset-condition} and \cref{lem:filtered} give a right
inverse $R_\pi:V\to\W$ with
\begin{equation}
 \mathcal IR_\pi=I_V,
 \qquad R_\pi\N_{J_k}\subseteq\W_{J_k},
 \quad 0\le k\le m.
 \label{eq:filtered-splitting}
\end{equation}

By \cref{eq:faces}, $\ker A_{J_k}=\N_{J_k}$ and
$\ker\mathcal D_{J_k}=\W_{J_k}$, so the second relation in
\cref{eq:filtered-splitting} gives
$\ker A_{J_k}\subseteq\ker(R_\pi^\top\mathcal D_{J_k}R_\pi)$.  By
\cref{lem:domination}, there are constants $c_{\pi,k}<\infty$ with
\begin{equation}
 \begin{aligned}
 R_\pi^\top\mathcal D_0R_\pi&\preceq c_{\pi,0}A_0,\\
 R_\pi^\top\mathcal D_{J_k}R_\pi
   &\preceq c_{\pi,k}A_{J_k},
   &&1\le k\le m.
 \end{aligned}
 \label{eq:cone-component-bounds}
\end{equation}
The first inequality uses $A_0\succ0$; if $A_{J_k}=0$, kernel preservation
also makes the left-hand side zero, so $c_{\pi,k}=0$.

Now let $\tau_{\pi_1}\ge\cdots\ge\tau_{\pi_m}\ge0$, set
$\tau_{\pi_{m+1}}=0$, and put $\beta_k=\tau_{\pi_k}-\tau_{\pi_{k+1}}\ge0$.
These coefficients give the telescoping identities
\begin{equation}
 \begin{aligned}
 A(\boldsymbol\tau)
   &=A_0+\sum_{k=1}^m\beta_kA_{J_k},\\
 \mathcal D(\boldsymbol\tau)
   &=\mathcal D_0+\sum_{k=1}^m\beta_k\mathcal D_{J_k},
 \end{aligned}
 \label{eq:cone-telescoping}
\end{equation}
where $\mathcal D(\boldsymbol\tau)=\operatorname{diag}_iA_i(\boldsymbol\tau)$.
With $C_\pi=\max_{0\le k\le m}c_{\pi,k}$,
\crefrange{eq:cone-component-bounds}{eq:cone-telescoping} give
\begin{equation}
 \norm{R_\pi v}_{\mathcal D(\boldsymbol\tau)}^2
 \le C_\pi\norm v_{A(\boldsymbol\tau)}^2
 \qquad\text{for every }v\in V,
 \label{eq:cone-stability}
\end{equation}
so $C_{\rm sd}(A(\boldsymbol\tau))\le C_\pi$ on this ordering cone.  The $m!$
cones cover the orthant and $C_{\rm ov}(A(\boldsymbol\tau))\le s$ by
\cref{prop:identity}, so
\begin{equation}
 \kappa\,\left(P(\boldsymbol\tau)^{-1}A(\boldsymbol\tau)\right)
 \le s\max_\pi C_\pi
 \label{eq:cone-orthant}
\end{equation}
throughout, proving \cref{eq:orthant-robustness}.
\end{proof}

The stable-decomposition constant on an active ray has the following
asymptotic form.  The resulting dichotomy also rules out intermediate powers
of $t$ in the condition number.

\begin{theorem}[Asymptotic smallest-eigenvalue rate]\label{thm:rate}
In the setting of \cref{eq:active-ray},
\begin{equation}
 \begin{aligned}
 \lim_{t\to\infty}\frac{C_{\rm sd}(A(t))}{t}&=\gamma_J,\\
 \lim_{t\to\infty}t\,\lambda_{\min}\,\left(P(t)^{-1}A(t)\right)
 &=\gamma_J^{-1},
 \end{aligned}
 \label{eq:sharp-rate}
\end{equation}
with the convention $\gamma_J^{-1}=\infty$ when $\gamma_J=0$.  Exactly one of
the following holds.
\begin{enumerate}
 \item $\N_J^{\mathrm{loc}}=\N_J$, and
 $\sup_{t\ge0}\kappa(P(t)^{-1}A(t))<\infty$.
 \item $\N_J^{\mathrm{loc}}\ne\N_J$, and
 \begin{equation}
  \gamma_J\le\liminf_{t\to\infty}\frac{\kappa(P(t)^{-1}A(t))}{t}
  \le\limsup_{t\to\infty}\frac{\kappa(P(t)^{-1}A(t))}{t}
  \le s\gamma_J .
  \label{eq:rate-bracket}
 \end{equation}
\end{enumerate}
\end{theorem}

\begin{proof}
Write $\mathcal D(t)=\bar{\mathcal D}+t\mathcal D_J$ with
$\bar{\mathcal D}=\operatorname{diag}_i(I_i^\top\bar AI_i)$; the proof of
\cref{thm:necessary} gives $C_{\rm sd}(A(t))\ge t\gamma_J$.

For the reverse inequality, let $S_J$ be the linear minimizer of
\cref{lem:obstruction} and $R$ any decomposition operator.  Because
$\bar A\succ0$, finite $\alpha,a,b$ exist with
$\norm{S_Jz}_{\bar{\mathcal D}}^2\le\alpha\norm z_{\bar A}^2$,
$\norm{Rz}_{\bar{\mathcal D}}^2\le a\norm z_{\bar A}^2$, and
$\norm{Rz}_{\mathcal D_J}^2\le b\norm z_{\bar A}^2$ for all $z\in V$.  Let
$\N_J^{\perp}$ be the $\bar A$-orthogonal complement of $\N_J$; since
$\ker A_J=\N_J$, the form $A_J$ is positive definite there, so
$\norm z_{A_J}^2\ge\mu\norm z_{\bar A}^2$ on $\N_J^{\perp}$ for some $\mu>0$.
Splitting $v=v_0+v_1$ with $v_0\in\N_J$ and $v_1\in\N_J^{\perp}$,
\begin{equation}
 \norm v_{A(t)}^2
 =\norm{v_0}_{\bar A}^2+\norm{v_1}_{\bar A}^2+t\norm{v_1}_{A_J}^2
 \ge\norm{v_0}_{\bar A}^2+(1+t\mu)\norm{v_1}_{\bar A}^2 .
 \label{eq:rate-denominator}
\end{equation}
Now $w=S_Jv_0+Rv_1\in\mathfrak D(v)$.  For $\epsilon>0$ and $M\succeq0$, the
Cauchy--Schwarz inequality gives
$\norm{x+y}_M^2\le(1+\epsilon)\norm x_M^2+(1+\epsilon^{-1})\norm y_M^2$.
Apply this inequality with $M=\mathcal D(t)$ and use
$\norm{S_Jv_0}_{\mathcal D(t)}^2\le(\alpha+t\gamma_J)\norm{v_0}_{\bar A}^2$
and $\norm{Rv_1}_{\mathcal D(t)}^2\le(a+tb)\norm{v_1}_{\bar A}^2$.  Dividing
the result by \cref{eq:rate-denominator} yields
\begin{equation}
 C_{\rm sd}(A(t))
 \le(1+\epsilon)(\alpha+t\gamma_J)
   +(1+\epsilon^{-1})\frac{a+tb}{1+t\mu}.
\end{equation}
Dividing by $t$ and letting $t\to\infty$ gives
$\limsup_tC_{\rm sd}(A(t))/t\le(1+\epsilon)\gamma_J$ for arbitrary
$\epsilon>0$, so $\lim_tC_{\rm sd}(A(t))/t=\gamma_J$, and
\cref{eq:sharp-rate} follows from
$\lambda_{\min}(P(t)^{-1}A(t))=C_{\rm sd}(A(t))^{-1}$.

For the dichotomy, if $\N_J^{\mathrm{loc}}=\N_J$ apply
\cref{thm:characterization} to the one-penalty family with base $\bar A\succ0$
and constraint $K_J$: its only kernel condition is
$\N_J^{\mathrm{loc}}=\N_J$, so the condition numbers are bounded on
$[0,\infty)$.  Otherwise $\gamma_J>0$ by \cref{eq:rate-positive}, and
\cref{eq:rate-bracket} follows from $\kappa=C_{\rm sd}C_{\rm ov}$ with
$1\le C_{\rm ov}\le s$ in \cref{prop:identity}.
\end{proof}

A single generalized eigenvalue therefore determines the leading term both of
$C_{\rm sd}(A(t))$ and of the reciprocal smallest eigenvalue.  For the
condition number itself the overlap constant intervenes, and since it is only
known to lie between $1$ and $s$, \cref{eq:rate-bracket} determines the order
of growth but not, in general, its coefficient.  \Cref{tab:sharp-rate}
compares these asymptotic results with computed spectra, and
\cref{sec:fe-negative} gives the corresponding finite-element calculation.

Decomposing the intersection of all the kernels is not sufficient.  The next
example decomposes $\N_{\{1,2\}}$ exactly, fails only on the singleton
$J=\{1\}$, and has a condition number that grows linearly; a
Scott--Vogelius realization of the same subset deficiency appears in
\cref{sec:fe-three}.

\begin{example}[A deficient single-penalty kernel]\label{ex:counterexample}
Let $e_1,e_2,e_3$ be the coordinate vectors in $\R^3$, and set
$A_0=I_{\R^3}$, $K_1=e_1^\top$, and $K_2=e_2^\top$, with subspaces
$V_1=\operatorname{span}\{e_3\}$,
$V_2=\operatorname{span}\{(e_1+e_2)/\sqrt2\}$, and
$V_3=\operatorname{span}\{(e_1-e_2)/\sqrt2\}$.
The subspaces span $\R^3$ and decompose
$\N_{\{1,2\}}=\operatorname{span}\{e_3\}$, but they do not decompose
$\N_{\{1\}}=\operatorname{span}\{e_2,e_3\}$.  For $\tau_1=t\ge1$ and
$\tau_2=1$, direct calculation gives
\begin{equation}
 P(t,1)^{-1}A(t,1)
 =\operatorname{diag}\,\left(\frac{2(t+1)}{t+3},
                              \frac{4}{t+3},1\right)
\end{equation}
in the basis $(e_1,e_2,e_3)$.  Hence
\begin{equation}
 \kappa(P(t,1)^{-1}A(t,1))=\frac{t+1}{2}.
 \label{eq:counterexample}
\end{equation}
Decomposing the full joint kernel therefore does not prevent deterioration
along a ray associated with a proper subset of the penalties.
\end{example}

The subset conditions need not yield a single decomposition operator that is
stable in every individual penalty energy.  The ordering-cone construction in
\cref{thm:characterization} accommodates this possibility.

\begin{example}[Ordering-dependent stability without a common splitting]
\label{ex:ordering-dependent}
Let $V=\R^2$, $A_0=I_{\R^2}$, and
\begin{equation}
 K_1=\begin{bmatrix}0&1\end{bmatrix},\qquad
 K_2=\begin{bmatrix}1&0\end{bmatrix},\qquad
 K_3=\begin{bmatrix}1&-1\end{bmatrix}.
\end{equation}
Take $V_1=V_2=V_3=\R$ with
$I_1a=ae_1$, $I_2b=be_2$, and $I_3c=c(e_1+e_2)$.  Each singleton kernel is
the range of the matching injection, and every joint kernel containing two or
more penalties is zero.  Thus \cref{eq:all-subset-condition} holds.

No common decomposition operator can preserve all three singleton kernels.
Indeed, kernel preservation would require $R(e_1)=(1,0,0)^\top$,
$R(e_2)=(0,1,0)^\top$, and $R(e_1+e_2)=(0,0,1)^\top$, which contradicts
linearity.  Nevertheless the exact additive preconditioner is uniformly
bounded: with $a=1+\tau_2+\tau_3$, $b=1+\tau_1+\tau_3$, and
$d=2+\tau_1+\tau_2$, direct calculation gives
\begin{equation}
 \begin{aligned}
 \operatorname{tr}(P^{-1}A)&=3,\\
 \det(P^{-1}A)
   &=2+\frac{2\tau_3(1+\tau_1)(1+\tau_2)}{abd}\ge2.
 \end{aligned}
 \label{eq:three-penalty-example}
\end{equation}
Since $P^{-1}A$ is similar to the SPD matrix $A^{1/2}P^{-1}A^{1/2}$, its two
eigenvalues are real and positive; their sum is $3$, and their product is at
least $2$ by \cref{eq:three-penalty-example}, so both lie in $[1,2]$ and
\begin{equation}
 \kappa\,\left(P(\boldsymbol\tau)^{-1}A(\boldsymbol\tau)\right)\le2
 \qquad\text{for every }\boldsymbol\tau\in[0,\infty)^3.
 \label{eq:three-penalty-bound}
\end{equation}
Thus the all-subset characterization is strictly weaker than the existence of
one common splitting that is stable in every component energy.
\end{example}

\subsection{How many kernel tests are needed}\label{sec:irredundance}

\Cref{eq:all-subset-condition} is a family of $2^m-1$ conditions, which invites
the question of how many are needed.  Two subsets give the same test when they
share a joint kernel, and only then.

\begin{proposition}[Reduction to distinct kernels]\label{prop:distinct}
The space $\N_J^{\mathrm{loc}}$ depends on $J$ only through the subspace
$\N_J$.  Hence the conditions in \cref{eq:all-subset-condition} are indexed by
\begin{equation}
 \mathcal K=\{\N_J:\varnothing\ne J\subseteq[m]\},
 \label{eq:distinct-kernels}
\end{equation}
the family of intersections of $\ker K_1,\ldots,\ker K_m$, and the number of
distinct tests is $\lvert\mathcal K\rvert\le2^m-1$.
\end{proposition}

\begin{proof}
By \cref{eq:faces}, $\N_J^{\mathrm{loc}}=\sum_i(I_iV_i\cap\N_J)$ is determined
by the subspace $\N_J$.  Since $\N_{J\cup J'}=\N_J\cap\N_{J'}$, the family
$\mathcal K$ is closed under intersection and is generated by the $m$
singleton kernels.
\end{proof}

Nested or repeated kernels can reduce the count substantially.  Without
additional structure, however, no further reduction is possible: each of the
$2^m-1$ tests can be the sole point of failure.

\begin{theorem}[Irredundance of the subset tests]\label{thm:irredundant}
Let $m\ge1$ and let $J_0\subseteq[m]$ be nonempty.  There is a family
\cref{eq:family} with $\dim V=\lvert J_0\rvert+1$, $A_0=I_V$, and
$s=2\lvert J_0\rvert$ one-dimensional subspaces spanning $V$, such that
\begin{equation}
 \N_J^{\mathrm{loc}}=\N_J
 \ \ \text{for every nonempty }J\ne J_0,
 \qquad
 \N_{J_0}^{\mathrm{loc}}\ne\N_{J_0}.
 \label{eq:irredundant}
\end{equation}
Consequently no condition in \cref{eq:all-subset-condition} is implied by the
remaining ones.  Taking $J_0=[m]$ makes the $2^m-1$ kernels distinct, so the
bound in \cref{prop:distinct} is attained.
\end{theorem}

\begin{proof}
Write $p=\lvert J_0\rvert$ and let $V=\R^{p+1}$ have the orthonormal basis
$\{z\}\cup\{u_j:j\in J_0\}$.  Set $A_0=I_V$, define $K_j=u_j^\top$ for
$j\in J_0$, and define $K_k=z^\top$ for $k\in[m]\setminus J_0$, identifying
each vector with the linear functional given by the Euclidean inner product.
Take the $2p$ one-dimensional subspaces spanned by $u_j$ and by $z+u_j$ for
$j\in J_0$.  They span $V$, since they contain every
$u_j$ and hence $z=(z+u_j)-u_j$.

Every $\N_J$ is spanned by a subset of the basis.  Writing $J'=J\cap J_0$,
\begin{equation}
 \N_J=
 \begin{cases}
  \operatorname{span}(\{z\}\cup\{u_k:k\in J_0\setminus J\}),
    &J\subseteq J_0,\\[2pt]
  \operatorname{span}\{u_k:k\in J_0\setminus J'\},
    &J\not\subseteq J_0 .
 \end{cases}
\end{equation}
A coordinate subspace of this basis contains $u_k$ exactly when
$k\in J_0\setminus J'$, and contains $z+u_k$ exactly when it contains both
$z$ and $u_k$.

If $J\subsetneq J_0$ then $J_0\setminus J\ne\varnothing$, so
$\N_J^{\mathrm{loc}}$ contains $u_k$ and $z+u_k$ for every
$k\in J_0\setminus J$, hence also $z$, giving $\N_J^{\mathrm{loc}}=\N_J$.  If
$J\not\subseteq J_0$ then $\N_J$ is spanned by the $u_k$ with
$k\in J_0\setminus J'$, each spanning a prescribed subspace, so again
$\N_J^{\mathrm{loc}}=\N_J$.  Finally $\N_{J_0}=\operatorname{span}\{z\}$
contains no prescribed subspace, so $\N_{J_0}^{\mathrm{loc}}=\{0\}$, proving
\cref{eq:irredundant}.

For $J_0=[m]$ the displayed formula gives
$\N_J=\operatorname{span}(\{z\}\cup\{u_k:k\notin J\})$ for every nonempty
$J\subseteq[m]$, and these $2^m-1$ subspaces are pairwise distinct.
\end{proof}

The construction in \cref{thm:irredundant} is evaluated for $m\le4$ in
\cref{sec:mac-several}.

Taking $m=2$ and $J_0=\{1,2\}$ in \cref{thm:irredundant} gives the opposite
failure pattern to \cref{ex:counterexample}: in $\R^3$ with $K_j=u_j^\top$ and
the four subspaces spanned by $u_1,u_2,z+u_1,z+u_2$, both singleton tests pass
while the pair test fails.  Neither the singleton tests nor the
full-intersection test can therefore be dispensed with in favor of the other,
and in both cases \cref{thm:rate} converts the failure into linear growth along
the corresponding ray.

\section{Computable bounds and the kernel lattice}\label{sec:quantitative}

\Cref{thm:characterization} decides robustness but leaves the resulting bound
implicit in its proof.  We now make it explicit.  The constants produced by the
ordering-cone argument are generalized eigenvalues and are therefore
computable, and a structural condition on the kernels determines when the $m!$
cone constructions reduce to one.

\Needspace{6\baselineskip}
\subsection{A computable bound on each ordering cone}

\begin{corollary}[Computable orthant bound]\label{cor:cone}
Assume \cref{eq:all-subset-condition}.  For each permutation $\pi$ of $[m]$
let $R_\pi$ satisfy \cref{eq:filtered-splitting}, write
$J_k=\{\pi_1,\ldots,\pi_k\}$, and let
\begin{equation}
 \begin{aligned}
 c_{\pi,0}&=\min\{c:R_\pi^\top\mathcal D_0R_\pi\preceq cA_0\},\\
 c_{\pi,k}&=\min\{c:R_\pi^\top\mathcal D_{J_k}R_\pi\preceq cA_{J_k}\},
   &&1\le k\le m .
 \end{aligned}
 \label{eq:cone-constants}
\end{equation}
With $q=\max_{0\le\ell\le m}q_\ell$ from \cref{eq:overlap},
\begin{equation}
 \sup_{\boldsymbol\tau\in[0,\infty)^m}
 \kappa\,\left(P(\boldsymbol\tau)^{-1}A(\boldsymbol\tau)\right)
 \le q\,\max_\pi\ \max_{0\le k\le m}c_{\pi,k}.
 \label{eq:cone-bound}
\end{equation}
For a nonzero denominator in \cref{eq:cone-constants}, the corresponding
constant is the largest generalized eigenvalue on the denominator's range; if
the denominator vanishes, the constant is zero.  Thus the right-hand side of
\cref{eq:cone-bound} is computable; it is evaluated in \cref{sec:mac-several}.
\end{corollary}

\begin{proof}
Each $c_{\pi,k}$ is finite by \cref{eq:cone-component-bounds}, and
\cref{lem:domination} identifies it with a largest generalized eigenvalue.  On
the cone $\tau_{\pi_1}\ge\cdots\ge\tau_{\pi_m}\ge0$, the telescoping identities
\cref{eq:cone-telescoping} have nonnegative coefficients $\beta_k$, so
\cref{eq:cone-constants} gives
$\norm{R_\pi v}_{\mathcal D(\boldsymbol\tau)}^2\le
 (\max_kc_{\pi,k})\norm v_{A(\boldsymbol\tau)}^2$, whence
$C_{\rm sd}(A(\boldsymbol\tau))\le\max_kc_{\pi,k}$ there.  For the overlap
constant, \cref{eq:overlap} and linearity of the energy in
$\boldsymbol\tau$ give, for every $w\in\W$,
\begin{equation}
 \norm{\mathcal Iw}_{A(\boldsymbol\tau)}^2
 =\sum_{\ell=0}^m\theta_\ell\norm{\mathcal Iw}_{A_\ell}^2
 \le\sum_{\ell=0}^m\theta_\ell q_\ell\norm w_{\mathcal D_\ell}^2
 \le q\norm w_{\mathcal D(\boldsymbol\tau)}^2 ,
\end{equation}
with $\theta_0=1$ and $\theta_j=\tau_j$, so $C_{\rm ov}\le q$ on the whole
orthant.  The $m!$ cones cover the orthant, and \cref{prop:identity} converts
the two bounds into \cref{eq:cone-bound}.
\end{proof}

When a common splitting is available and is used on every cone, the resulting
bound is no larger than the classical one.

\begin{proposition}[Comparison with the common-splitting bound]
\label{prop:refines}
Suppose a decomposition operator $R$ satisfies \cref{eq:componentstable} with
constants $C_\ell$.  Then $R$ is admissible in \cref{eq:filtered-splitting}
for every $\pi$, and the choice $R_\pi=R$ gives
$\max_\pi\max_kc_{\pi,k}\le\max_\ell C_\ell$.  Hence \cref{eq:cone-bound} is
at most the bound $Cq$ of \cref{eq:uniformbound}, and \cref{thm:uniform} is
the special case in which one splitting serves every ordering.
\end{proposition}

\begin{proof}
As noted after \cref{eq:componentstable}, stability in the energy of a
semidefinite $A_\ell$ forces $I_iR_iv\in\ker A_\ell$ whenever
$v\in\ker A_\ell$.  For $v\in\N_{J_k}=\bigcap_{\ell\in J_k}\ker A_\ell$ this
gives $I_iR_iv\in\N_{J_k}$ for every $i$, that is, $Rv\in\W_{J_k}$; so $R$
satisfies \cref{eq:filtered-splitting}.  Summing
$R^\top\mathcal D_\ell R\preceq C_\ell A_\ell$ over $\ell\in J_k$ gives
$R^\top\mathcal D_{J_k}R\preceq(\max_{\ell\in J_k}C_\ell)A_{J_k}$, so
$c_{\pi,k}\le\max_\ell C_\ell$ for every $\pi$ and $k$.
\end{proof}

\subsection{When one common splitting exists}

Let $\mathcal L$ be the lattice of subspaces of $V$ generated by
$\{\N_J:\varnothing\ne J\subseteq[m]\}$ together with $\{0\}$ and $V$, under
sum and intersection.  Recall that $\mathcal L$ is \emph{distributive} when
$X\cap(Y+Z)=(X\cap Y)+(X\cap Z)$ for all $X,Y,Z\in\mathcal L$.
Distributivity provides a single basis of $V$ adapted simultaneously to every
member of $\mathcal L$, which permits a common splitting to be defined kernel
by kernel.  Moreover, $\mathcal L$ is then finite: it has finitely
many generators, and distributivity reduces every lattice expression to a join
of meets of subsets of them~\cite{gratzer2011lattice}.  The following
finite-lattice result therefore applies.

\begin{lemma}[Adapted basis]\label{lem:adapted}
Let $\mathcal L$ be a finite distributive lattice of subspaces of $V$
containing $\{0\}$ and $V$.  Then $V$ has a basis $B$ such that every
$U\in\mathcal L$ is spanned by
$B\cap U$~\cite[Lem.~4.4.2]{loday2012algebraic}; see also
\cite{harrison2014reflexivity}.
\end{lemma}

\begin{theorem}[Common splitting for a distributive kernel lattice]
\label{thm:distributive}
Assume \cref{eq:all-subset-condition} and that $\mathcal L$ is distributive.
Then there is a decomposition operator $R$ with
\begin{equation}
 R\,\N_J\subseteq\W_J
 \qquad\text{for every nonempty }J\subseteq[m].
 \label{eq:common-splitting}
\end{equation}
Consequently \cref{eq:componentstable} holds with finite constants $C_\ell$,
and \cref{thm:uniform} bounds the condition number by $Cq$ over the entire
parameter orthant.
\end{theorem}

\begin{proof}
Let $B$ be a basis adapted to $\mathcal L$ as in \cref{lem:adapted}.  For
$b\in B$ set $J^\ast(b)=\bigcup\{J:\varnothing\ne J\subseteq[m],\
b\in\N_J\}$, with $J^\ast(b)=\varnothing$ if no such $J$ exists.  Because
$\N_{J\cup J'}=\N_J\cap\N_{J'}$, we have $b\in\N_{J^\ast(b)}$ whenever
$J^\ast(b)\ne\varnothing$.  By \cref{eq:all-subset-condition} the map
$\mathcal I:\W_{J^\ast(b)}\to\N_{J^\ast(b)}$ is onto, so we may choose
$Rb\in\W_{J^\ast(b)}$ with $\mathcal IRb=b$; for $J^\ast(b)=\varnothing$
choose any $Rb$ with $\mathcal IRb=b$.  Extending linearly defines $R$ with
$\mathcal IR=I_V$.

Let $J$ be nonempty and $v\in\N_J$.  Since $B$ is adapted and
$\N_J\in\mathcal L$, we may write $v=\sum c_bb$ over those $b\in B$ that lie
in $\N_J$.  For each such $b$ we have $J\subseteq J^\ast(b)$, hence
$\N_{J^\ast(b)}\subseteq\N_J$ and therefore
$\W_{J^\ast(b)}\subseteq\W_J$, so $Rb\in\W_J$.  Thus $Rv\in\W_J$, which is
\cref{eq:common-splitting}.

Finally, \cref{eq:common-splitting} for $J=\{\ell\}$ gives
$\ker A_\ell\subseteq\ker(R^\top\mathcal D_\ell R)$ for $\ell\ge1$, and
$A_0\succ0$ handles $\ell=0$; each $C_\ell$ is therefore finite by
\cref{lem:domination}.
\end{proof}

Two frequently occurring situations are covered automatically.

\begin{corollary}\label{cor:two}
If $m=2$, then \cref{eq:all-subset-condition} implies the existence of a
common splitting \cref{eq:common-splitting}.
\end{corollary}

\begin{proof}
Here $\mathcal L$ is generated by $\N_{\{1\}}$ and $\N_{\{2\}}$, because
$\N_{\{1,2\}}=\N_{\{1\}}\cap\N_{\{2\}}$.  The sublattice generated by two
elements $X$ and $Y$ of any lattice is $\{X\cap Y,X,Y,X+Y\}$, which is a
chain when $X$ and $Y$ are comparable and otherwise a four-element Boolean
lattice; both are distributive.  Adjoining a bottom and a top element
preserves distributivity, since every instance of the distributive law in
which one argument is $\{0\}$ or $V$ reduces to an identity.  Apply
\cref{thm:distributive}.
\end{proof}

\begin{corollary}\label{cor:commuting}
If the projectors onto $\ker K_1,\ldots,\ker K_m$ that are orthogonal in one
fixed inner product commute pairwise, or if these kernels form a chain, then
$\mathcal L$ is distributive and \cref{eq:common-splitting} holds under
\cref{eq:all-subset-condition}.
\end{corollary}

\begin{proof}
Commuting orthogonal projectors are simultaneously diagonalizable, so each
$\ker K_j$ is spanned by a subset of one orthonormal basis.  Sums and
intersections of such coordinate subspaces are again coordinate subspaces,
and correspond to unions and intersections of index sets, so $\mathcal L$
embeds into the Boolean lattice of subsets, which is distributive.  A chain is
distributive.  Apply \cref{thm:distributive}.
\end{proof}

\begin{remark}\label{rem:m3}
The all-subset condition need not yield a common splitting.  In
\cref{ex:ordering-dependent} the three singleton kernels are distinct lines
in $\R^2$.  Together with $\{0\}$ and $V$, they form the five-element lattice
$M_3$.  Their pairwise intersections are zero and their pairwise sums equal
$V$, so distributivity fails because
\begin{equation}
 \N_{\{1\}}\cap(\N_{\{2\}}+\N_{\{3\}})=\N_{\{1\}},
 \qquad
 (\N_{\{1\}}\cap\N_{\{2\}})+(\N_{\{1\}}\cap\N_{\{3\}})=\{0\}.
 \label{eq:m3-failure}
\end{equation}
That example has no common splitting, although \cref{cor:cone} still supplies
a finite orthant bound.  By \cref{cor:two}, a nondistributive kernel lattice
cannot occur with fewer than three penalties; \cref{sec:fe-three} reports a
numerical Scott--Vogelius instance with three regional weights.
\end{remark}

\Needspace{6\baselineskip}
\section{Finite tests and bounds}\label{sec:construction}

\Cref{alg:checks} summarizes the finite-dimensional kernel test and bound
construction.  In exact arithmetic, a positive deficiency identifies a ray
on which \cref{thm:necessary} proves deterioration at the rate $\gamma_J$; if
no deficiency occurs, \cref{thm:characterization} gives fixed-dimensional
robustness.  Mesh independence additionally requires the constants to remain
uniform under refinement, which rank agreement on individual meshes does not
imply.

The resulting design rule differs from the familiar ones.  Classical coarse
spaces approximate low-energy error~\cite{toselli2005domain} and spectral
enrichment selects local modes through generalized
eigenproblems~\cite{spillane2014geneo}, whereas a space built for a dominant
semidefinite penalty must also represent that penalty's
kernel~\cite{farrell2021reynolds}.  The kernel comparison in
\cref{alg:checks} evaluates this property once per distinct joint kernel, and
\cref{eq:finite-inequalities} supplies the associated common-splitting bound.

\paragraph{The one-penalty splitting}
For one constraint, the construction is explicit.  Put
$\N=\ker K$ and $\W_{\N}=\{(v_i)\in\W:KI_iv_i=0\ \text{for every }i\}$, let
$\Pi_{\N}:V\to\N$ be the $A_0$-orthogonal projector, and note that
$\N^{\mathrm{loc}}=\N$ makes $\mathcal I:\W_{\N}\to\N$ onto, with a linear
right inverse $E_{\N}$.  The $\mathcal D_0$-energy-minimizing right inverse of
$\mathcal I$ on $V$ is
\begin{equation}
 \begin{aligned}
 E_0&=\mathcal D_0^{-1}\mathcal I^\top
     (\mathcal I\mathcal D_0^{-1}\mathcal I^\top)^{-1},\\
 R&=E_{\N}\Pi_{\N}+E_0(I_V-\Pi_{\N}).
 \end{aligned}
 \label{eq:splitter}
\end{equation}
The inverse exists because $\mathcal D_0$ is SPD and $\mathcal I$ onto, and the
normal equations for the constrained minimum give the first formula; hence
$\mathcal IR=I_V$ and $R$ maps every global null vector to components in the
local kernels.  For each component the least constants are determined by
\begin{equation}
 \begin{aligned}
 R^\top\mathcal D_\ell R
   &\preceq C_\ell A_\ell,\\
 \mathcal I^\top A_\ell\mathcal I
   &\preceq q_\ell\mathcal D_\ell.
 \end{aligned}
 \label{eq:finite-inequalities}
\end{equation}
By \cref{lem:domination}, generalized eigenvalues determine whether
\cref{eq:finite-inequalities} holds and, when it does, give the least
constants.

\paragraph{The general algorithm}
For several penalties, \cref{alg:checks} summarizes the characterization and
the ordering-cone bound.  Here $\operatorname{null}(M)$ denotes any matrix
whose columns form a basis of $\ker M$.  Semidefinite generalized eigenvalues
are evaluated on the orthogonal complement of the denominator's kernel, with
value zero when the denominator vanishes.

\begin{algorithm}[tbhp]
\caption{Joint-kernel test and computable orthant bound.}\label{alg:checks}
\begin{algorithmic}[1]
\REQUIRE $A_0$, $K_1,\ldots,K_m$, injective subspace maps
$I_1,\ldots,I_s$ whose ranges span $V$
\STATE Form $\mathcal I=[I_1\ \cdots\ I_s]$.
\STATE Form $A_\ell=K_\ell^\top K_\ell$ and
$\mathcal D_\ell=\operatorname{diag}_i(I_i^\top A_\ell I_i)$ for
$\ell\in[m]$, together with $A_0$ and $\mathcal D_0$.
\FOR{every nonempty $J\subseteq[m]$}
  \STATE Stack the matrices $K_j$, $j\in J$, to form $K_J$, and set
  $U_J=\operatorname{null}(K_J)$.
  \STATE Set $Z_{i,J}=I_i\operatorname{null}(K_JI_i)$ and
  $Z_J=[Z_{1,J}\ \cdots\ Z_{s,J}]$.
  \STATE Set $d_J=\dim\ker K_J-\operatorname{rank}(Z_J)$.
\ENDFOR
\IF{$d_J>0$ for some $J$}
  \RETURN the deficient sets $J$ and their deficiencies $d_J$.
\ENDIF
\FOR{$\ell=0,\ldots,m$}
  \STATE Compute the least $q_\ell$ satisfying
  $\mathcal I^\top A_\ell\mathcal I\preceq q_\ell\mathcal D_\ell$.
\ENDFOR
\FOR{every permutation $\pi$ of $[m]$}
  \STATE Set $J_k=\{\pi_1,\ldots,\pi_k\}$ and construct $R_\pi$ by the
  complement-and-lift procedure in \cref{lem:filtered}.
  \STATE Compute the least $c_{\pi,0},\ldots,c_{\pi,m}$ in
  \cref{eq:cone-constants}.
\ENDFOR
\RETURN $q=\max_\ell q_\ell$ and the orthant bound in
\cref{eq:cone-bound}.
\end{algorithmic}
\end{algorithm}

In exact arithmetic a positive $d_J$ proves linear deterioration along the ray
for $J$, whose rate follows from $G_J$ and \cref{eq:rate-constant} once the
remaining parameters are fixed.  In floating point the rank and kernel
comparisons require a stated tolerance and are not exact decisions.  Repeated
joint kernels need only one rank test by \cref{prop:distinct}, but without
further structure the worst-case count remains exponential
(\cref{thm:irredundant}), and the ordering-cone stage is factorial.  When
\cref{thm:distributive} applies, one common splitting and
\cref{eq:finite-inequalities} replace that stage.

\section{Computational study}\label{sec:study}

The computations serve three distinct purposes: they evaluate the algebraic
identities and implementations on finite-dimensional examples, compare the
asymptotic coefficient $\gamma_J$ with computed spectra, and examine the same
kernel mechanism in a finite-element W-cycle.  These calculations illustrate
the theoretical results but do not prove them.

On the staggered-grid model of \cref{sec:mac}, every constant can be formed
explicitly and the exact-additive theorems apply.  In \cref{sec:fe-exact}, the
preconditioner remains exactly additive, while the subspaces arise from a
Scott--Vogelius discretization.  The inexact W-cycle of \cref{sec:fe-cycle}
lies outside the theorem hypotheses and is therefore reported as separate
empirical evidence.  \Cref{tab:plan} summarizes the role of each experiment.

\begin{table}[tbhp]
\caption{Numerical role of each experiment.  The first two columns evaluate
the exact-additive theory on selected finite-dimensional problems; the
W-cycle provides evidence outside the theorem hypotheses.}
\label{tab:plan}
\centering
\small
\begin{tabular}{@{}p{.26\linewidth}p{.18\linewidth}p{.23\linewidth}
                  p{.17\linewidth}@{}}
\toprule
result or mechanism & staggered grid & exact additive, SV & inexact W-cycle\\
\midrule
kernel obstruction and rate
  & \cref{sec:mac-rate} & \cref{sec:fe-negative}; \cref{sec:fe-three} & n/a\\
subset characterization
  & \cref{sec:mac-several} & \cref{sec:regional-penalties} & n/a\\
bounds and lattice
  & \cref{sec:mac-several} & \cref{sec:fe-three} & n/a\\
kernel-preserving solver behavior
  & n/a & n/a & \cref{sec:fe-cycle}\\
\bottomrule
\end{tabular}
\end{table}

\subsection{A staggered-grid model}\label{sec:mac}

Every constant of \crefrange{sec:upper}{sec:quantitative} can be formed
explicitly on this model and evaluated by dense linear algebra.  Let $h=1/n$
and partition $\Omega=(0,1)^2$ into $n^2$ square cells using the
marker-and-cell (MAC) staggered grid~\cite{harlow1965mac}.  The velocity space
$V_h\cong\R^{N_h}$, $N_h=2n(n-1)$, holds the values
$u_{i,j}$ ($1\le i\le n-1$, $1\le j\le n$) and $v_{i,j}$
($1\le i\le n$, $1\le j\le n-1$) on vertical and horizontal interior faces
respectively, with values outside these index sets zero.

For $w=(u,v)\in V_h$, define the scaled cellwise divergence
$K_h=h\,\operatorname{div}_h\in\R^{n^2\times N_h}$ by
\begin{equation}
 (K_hw)_{i,j}=u_{i,j}-u_{i-1,j}+v_{i,j}-v_{i,j-1},
 \qquad 1\le i,j\le n.
 \label{eq:mac-divergence}
\end{equation}
Hence $\norm{K_hw}_2^2$ is the midpoint quadrature of the squared pointwise
discrete divergence, and the base matrix is the componentwise five-point graph
Dirichlet operator,
\begin{align}
 (A_0w)^u_{i,j}
 &=4u_{i,j}-u_{i-1,j}-u_{i+1,j}-u_{i,j-1}-u_{i,j+1}, \notag\\
 (A_0w)^v_{i,j}
 &=4v_{i,j}-v_{i-1,j}-v_{i+1,j}-v_{i,j-1}-v_{i,j+1}.
 \label{eq:mac-base}
\end{align}
The model operator is $A_h(\tau)=A_0+\tau K_h^\top K_h$ for $\tau\ge0$.

\paragraph{The discrete exact sequence}
Let $\Psi_h\cong\R^{(n-1)^2}$ contain scalar values at the interior grid
vertices, extended by zero to the boundary.  Define $C_h:\Psi_h\to V_h$ by
\begin{equation}
 \begin{aligned}
 (C_h\psi)^u_{i,j}&=\frac{\psi_{i,j}-\psi_{i,j-1}}{h},
   &&1\le i\le n-1,\quad 1\le j\le n,\\
 (C_h\psi)^v_{i,j}&=\frac{\psi_{i-1,j}-\psi_{i,j}}{h},
   &&1\le i\le n,\quad 1\le j\le n-1.
 \end{aligned}
 \label{eq:mac-curl}
\end{equation}
If $Q_h^0=\{q\in\R^{n^2}:\boldsymbol 1^\top q=0\}$, then
\begin{equation}
 0\longrightarrow\Psi_h\xrightarrow{\,C_h\,}V_h
 \xrightarrow{\,K_h\,}Q_h^0\longrightarrow0
 \label{eq:exactsequence}
\end{equation}
is exact: substitution gives $K_hC_h=0$, while $C_h\psi=0$ forces adjacent
vertex values to agree and hence $\psi=0$ by the zero boundary extension, and
$K_h^\top q=0$ forces $q$ constant.  So $C_h$ has rank $(n-1)^2$,
$\operatorname{rank}K_h=n^2-1$, and $\ker K_h=\range C_h$ has dimension
$(n-1)^2$.  We retain all $n^2$ rows of $K_h$; their one linear dependence
expresses $\range K_h=Q_h^0$.

\paragraph{Vertex patches}
For an interior vertex $(i,j)$, let
\begin{equation}
 V_{i,j}=\operatorname{span}
 \{e^u_{i,j},e^u_{i,j+1},e^v_{i,j},e^v_{i+1,j}\}.
\end{equation}
The restriction of $K_h$ to $V_{i,j}$ has rank three, and its null vector is
$C_he^\psi_{i,j}$.  Therefore
\begin{equation}
 \begin{aligned}
 V_{i,j}\cap\ker K_h
   &=\operatorname{span}\{C_he^\psi_{i,j}\},\\
 \ker K_h
   &=\sum_{i=1}^{n-1}\sum_{j=1}^{n-1}
     (V_{i,j}\cap\ker K_h).
 \end{aligned}
 \label{eq:patch-kernel}
\end{equation}
Every one-coordinate velocity subspace, by contrast, has trivial intersection
with $\ker K_h$.

\paragraph{Coarse spaces}
Suppose that $n$ is even and set $H=2h$.  Let
\begin{equation}
 J_\psi\in\R^{(n-1)^2\times(n/2-1)^2},\qquad
 J_q\in\R^{n^2\times(n/2)^2}
\end{equation}
denote, respectively, the bilinear interpolation matrix for zero-boundary
coarse vertex values and the tensor-product piecewise-linear interpolation
matrix for coarse cell-centered values with endpoint clamping.  With
$\operatorname{orth}(X)$ denoting a
matrix whose orthonormal columns span $\range X$, define
\begin{equation}
 \begin{aligned}
 Z_\psi&=\operatorname{orth}(C_hJ_\psi),&
 Z_q&=\operatorname{orth}(K_h^\top J_q),\\
 Z_c&=\operatorname{orth}[Z_\psi\ Z_q].
 \end{aligned}
 \label{eq:coarse}
\end{equation}
Thus $K_hZ_\psi=0$ and $\range Z_q\subseteq\range K_h^\top$, and the coarse
space is $V_c=\range Z_c$.  Each calculation is two-level, building $V_c$ from
the associated $2h$ grid and solving the coarse problem exactly.

\paragraph{Numerical setup}
The dense calculations use NumPy 2.5.2 and SciPy
1.18.1~\cite{harris2020numpy,virtanen2020scipy}.  Condition numbers are
computed from $P^{-1/2}A_hP^{-1/2}$, which is similar to $P^{-1}A_h$.  The
dimension and rank identities, reconstruction properties, and kernel
inclusions are evaluated at the reported tolerances.  For the preconditioned
conjugate-gradient (PCG) solves, true residuals are recomputed and the random
right-hand sides use seed 20270829.

\subsubsection{One penalty: the mechanism and its rate}
\label{sec:mac-rate}

For $n=12$ (264 unknowns), we evaluate the inverse penalty
$\rho=\tau^{-1}$ from $1$ to $10^{-10}$; \cref{fig:penalty} plots the
resulting condition numbers.  Write $A_\rho=A_h(\rho^{-1})$ and
$P_\rho^{-1}=P(\rho^{-1})^{-1}$.  Point subspaces fail to decompose
$\ker K_h$; accordingly, \cref{thm:necessary} implies unbounded growth, and
the computed condition number rises from $82.62$ to $2.56\times10^{11}$.
Vertex patches satisfy \cref{eq:patch-kernel}, but without a coarse space the
condition number still reaches $1728.03$ at the smallest $\rho$.  Adding the
coarse space keeps the computed condition numbers between $7.47$ and $18.47$
over the full sweep.

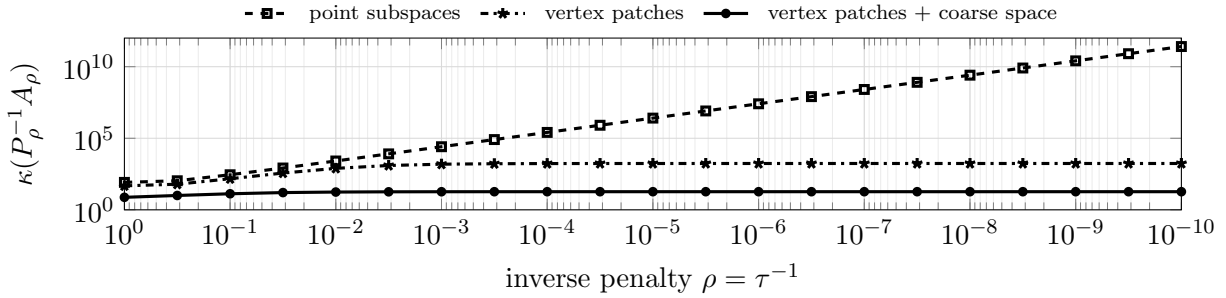
\begin{figure}[tbhp]
\centering
\begin{tikzpicture}[scale=.98,transform shape]
\begin{loglogaxis}[
 width=.96\linewidth,height=3.9cm,
 xlabel={inverse penalty $\rho=\tau^{-1}$},
 ylabel={$\kappa(P_\rho^{-1}A_\rho)$},
 xmin=1e-10,xmax=1,ymin=1,ymax=1e12,enlarge x limits=false,
 x dir=reverse,grid=both,minor grid style={gray!15},major grid style={gray!30},
 legend style={at={(axis description cs:.5,1.025)},anchor=south,draw=none,font=\scriptsize,
               legend columns=3,column sep=5pt},
 mark options={scale=.62},
]
\addplot[black,very thick,dashed,mark=square,
         mark options={solid,fill=white,scale=.75}]
 table[x=epsilon,y=point]{\PenaltySweepData};
\addlegendentry{point subspaces}
\addplot[black,very thick,dashdotted,mark=star,
         mark options={solid,fill=black,scale=1.05}]
 table[x=epsilon,y=vertex]{\PenaltySweepData};
\addlegendentry{vertex patches}
\addplot[black,very thick,mark=*] table[x=epsilon,y=cone]{\PenaltySweepData};
\addlegendentry{vertex patches $+$ coarse space}
\end{loglogaxis}
\end{tikzpicture}
\caption{Fixed-grid parameter sweep at $n=12$.  Point corrections violate the
kernel condition, vertex patches satisfy it without controlling coarse modes,
and the combined space adds the coarse basis of \cref{eq:coarse}.}
\label{fig:penalty}
\end{figure}

The two families in \cref{eq:coarse} address different regimes, and on seven
grids through $n=16$ neither is adequate alone.  At $n=16$ the range family
$Z_q$ improves the base regime from $82.44$ to $52.14$ at $\rho=1$ but leaves
$\rho=10^{-8}$ near $5.1\times10^3$, whereas the kernel family $Z_\psi$ reduces
that stiff value to $111.40$.  The combined space keeps both regimes below
$20$ on every grid.  This behavior is consistent with the requirement in
\cref{thm:uniform} that one decomposition be stable simultaneously in the
base and penalty energies.

\Cref{thm:rate} predicts the leading decay of the smallest eigenvalue.  For the
one-penalty family $A_h(\tau)$, the only nonempty subset is $J=\{1\}$, with
$\bar A=A_0$ and $\N_J=\ker K_h$.  We form
the matrix $G_J$ representing \cref{eq:kernel-obstruction}, compute
$\gamma_J$ from \cref{eq:rate-constant}, and compare
$\gamma_J^{-1}$ with $\tau\lambda_{\min}(P(\tau)^{-1}A_h(\tau))$ at
$\tau=10^{10}$ for both space families in \cref{tab:sharp-rate}.

\begin{table}[tbhp]
\caption{Asymptotic smallest-eigenvalue rate at $\tau=10^{10}$ on three of
the five computed grids.  Point subspaces have
$\N_J^{\mathrm{loc}}=\{0\}$, so $\gamma_J>0$ and \cref{thm:rate} predicts
$\tau\lambda_{\min}\to\gamma_J^{-1}$; vertex patches with the coarse space
satisfy the kernel condition, so $\gamma_J$ vanishes and $\lambda_{\min}$ stays
bounded away from zero.}
\label{tab:sharp-rate}
\centering
\small
\begin{tabular}{rrrrrrr}
\toprule
& & & \multicolumn{3}{c}{point subspaces} & vertex $+$ coarse\\
\cmidrule(lr){4-6}\cmidrule(lr){7-7}
$n$ & unknowns & $\dim\N_J$ & $\gamma_J$ & $\gamma_J^{-1}$
 & $\tau\lambda_{\min}$ & $\lambda_{\min}$\\
\midrule
4 & 24 & 9 & 1.0111 & 0.98904 & 0.98905 & 0.40876\\
8 & 112 & 49 & 3.1505 & 0.31741 & 0.31741 & 0.24202\\
12 & 264 & 121 & 6.5123 & 0.15355 & 0.15356 & 0.21656\\
\bottomrule
\end{tabular}
\end{table}

Predicted and observed values agree to at least four significant digits on
every grid, with $\dim\N_J^{\mathrm{loc}}=0$ throughout for the point
subspaces.  For the vertex patches with the coarse space the computed
$\gamma_J$ is at most $3.4\times10^{-15}$, which is consistent with zero at
floating-point precision, and $\lambda_{\min}$ remains above $0.21$.  The two
cases of \cref{thm:rate} are therefore also separated in the computed spectra.

\subsubsection{One penalty: computed bounds under refinement}

\Cref{fig:scale} extends the two-level calculation to $n=40$ (3,120 unknowns).
At $\tau=10^8$, the condition number changes from $19.17$ to $20.14$ between
$n=16$ and $40$, corresponding to a fitted log--log slope of $0.053$.  PCG
takes 52 iterations to reach a recomputed relative residual below
$2\times10^{-7}$ for $n=24,28,32,36,40$, and the condition number at $\tau=1$
stays below $9.78$.
For $n\le16$ we construct $R$ from \cref{eq:splitter} and compute its
component constants: at $n=16$,
$(C_0,C_1,q_0,q_1)=(4.865,1.344,3.242,4.000)$ gives $Cq\approx19.459$ against
an observed $19.166$ at $\tau=10^8$, with the individual constants in
\cref{tab:component-bounds}.

\begin{table}[tbhp]
\caption{Computed component-energy constants and the observed condition number
at $\tau=10^8$ on four of the seven computed grids.  Reconstruction and
kernel-preservation residuals are below $1.5\times10^{-14}$.}
\label{tab:component-bounds}
\centering
\small
\begin{tabular}{rrrrrrr}
\toprule
$n$ & $C_0$ & $C_1$ & $q_0$ & $q_1$ & $Cq$ & observed\\
\midrule
4  & 2.449 & 1.008 & 2.993 & 3.865 & 9.465 & 9.455\\
8  & 4.181 & 1.083 & 3.177 & 3.991 & 16.688 & 16.491\\
12 & 4.676 & 1.229 & 3.224 & 3.998 & 18.697 & 18.463\\
16 & 4.865 & 1.344 & 3.242 & 4.000 & 19.459 & 19.166\\
\bottomrule
\end{tabular}
\end{table}

Over all seven grids the computed bound exceeds the observed value by at
most $1.4\%$, so the estimate of \cref{thm:uniform} is close to sharp here.
The estimate is not mesh-uniform: $C_0$ increases across the sampled grids,
consistent with the fixed-mesh scope of the theorem.

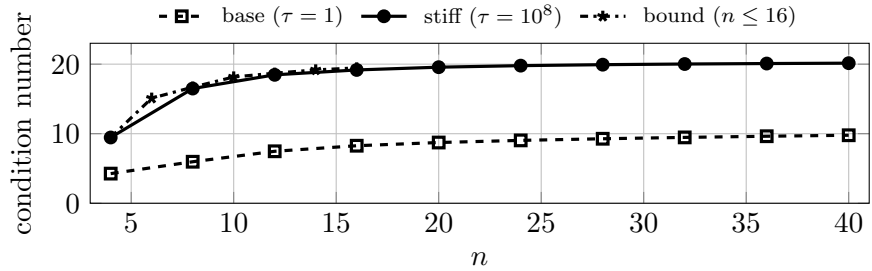
\begin{figure}[tbhp]
\centering
\begin{tikzpicture}
\begin{axis}[width=.72\linewidth,height=3.7cm,grid=major,
 xlabel={$n$},ylabel={condition number},xmin=3,xmax=41,ymin=0,ymax=23,
 legend style={at={(axis description cs:.5,1.03)},anchor=south,draw=none,
               font=\scriptsize,legend columns=3,column sep=4pt}]
\addplot[black,very thick,dashed,mark=square,
         mark options={solid,fill=white}]
 table[x=n,y=base_condition]{\ScalabilityData};
\addlegendentry{base ($\tau=1$)}
\addplot[black,very thick,mark=*] table[x=n,y=stiff_condition]{\ScalabilityData};
\addlegendentry{stiff ($\tau=10^8$)}
\addplot[black,very thick,dashdotted,mark=star,
         mark options={solid,fill=black,scale=1.05}]
 table[x=n,y=bound]{\OneParameterBoundData};
\addlegendentry{bound ($n\le16$)}
\end{axis}
\end{tikzpicture}
\caption{Mesh refinement for the staggered-grid problem.  Each calculation
uses an exact $2h$ coarse solve; the computed bound $Cq$ is shown for
$n\le16$.}
\label{fig:scale}
\end{figure}

For $n>16$, forming the dense $R$ is impractical, so \cref{fig:scale} reports
only the observed condition numbers.  They remain nearly constant through
$n=40$; the calculation supplies no asymptotic mesh bound.

\subsubsection{Several penalties: orthant bounds and the test family}
\label{sec:mac-several}

For even $n$, let $K_1$ contain the divergence rows for the first $n/2$ columns
of cells and let $K_2$ contain the rows for the remaining columns.  Let
$\Pi_j$ be the $A_0$-orthogonal projector onto
$\ker K_j$.  At $n=8$,
$\norm{\Pi_1\Pi_2-\Pi_2\Pi_1}_2=0.222$, so these two kernel projectors do not
commute.

The global and assembled local kernel dimensions agree for
$J=\varnothing,\{1\},\{2\},\{1,2\}$, with pairs $112/112$, $80/80$, $80/80$,
and $49/49$.  Over the sixteen combinations of
$\rho_j\in\{1,10^{-3},10^{-6},10^{-9}\}$, vertex patches alone reach $387.90$
while adding both coarse basis families reduces the largest observed value to
$32.17$.

To evaluate the sufficient condition beyond those sixteen points, let $U_J$
have orthonormal columns spanning $\N_J$ and set
\begin{equation}
 \begin{aligned}
 E_1&=\operatorname{orth}\bigl((I-U_{\{1,2\}}U_{\{1,2\}}^\top)U_{\{1\}}\bigr),\\
 E_2&=\operatorname{orth}\bigl((I-U_{\{1,2\}}U_{\{1,2\}}^\top)U_{\{2\}}\bigr),\\
 E_0&=\operatorname{null}\bigl([U_{\{1,2\}}\ E_1\ E_2]^\top\bigr).
 \end{aligned}
 \label{eq:two-penalty-splitting}
\end{equation}
Lifting $U_{\{1,2\}}$, $E_1$, and $E_2$ through the matching assembled local
kernels, and $E_0$ through the $\mathcal D_0$-minimal right inverse, defines
$R$ on the basis $[U_{\{1,2\}}\ E_1\ E_2\ E_0]$, which is adapted to the
lattice generated by $\N_{\{1\}}$ and $\N_{\{2\}}$.  By \cref{cor:two} the
three kernel equalities imply a common splitting exists, and direct enumeration
gives a six-member distributive lattice.  At $n=8$ the four block dimensions
are $49$, $31$, $31$, and $1$, and direct calculation gives
\begin{equation}
 (C_0,C_1,C_2)=(9.011,6.918,6.918),\qquad
 (q_0,q_1,q_2)=(3.177,4.824,4.824).
\end{equation}
The reconstruction residual is $9.2\times10^{-15}$ and the two
kernel-preservation residuals are below $7.6\times10^{-15}$, giving
\begin{equation}
 Cq\approx43.470.
 \label{eq:two-parameter-bound}
\end{equation}
Because the calculation has not been validated by interval arithmetic,
\cref{eq:two-parameter-bound} is a numerical estimate of the uniform bound
implied by the exact component inequalities; it is $1.351$ times the largest
sampled condition number, $32.167$.

For \cref{ex:ordering-dependent}, where \cref{rem:m3} shows no common splitting
exists, we build the six filtered right inverses of
\cref{eq:filtered-splitting}, one per ordering, and evaluate
\cref{eq:cone-bound}.  The overlap constant is $q=2$, the largest cone constant
is $\max_\pi\max_k c_{\pi,k}=1.9250$, and the resulting bound is $3.8499$.
Since the exact supremum is $2$ by \cref{eq:three-penalty-bound}, the computable
bound exceeds it by a factor of $1.925$ in a case where \cref{thm:uniform}
does not apply.  For a three-penalty problem in $\R^4$ with a
distributive kernel lattice, the adapted-basis construction of
\cref{thm:distributive} yields a common splitting with reconstruction residual
$3.3\times10^{-16}$, and \cref{eq:cone-bound} gives $3.0297$ against a worst
sampled condition number of $2.9443$ over $216$ orthant points.

For \cref{thm:irredundant} we assemble the stated instance for every nonempty
$J_0\subseteq[m]$ with $m\le4$, $26$ instances in all, and evaluate the
$2^m-1$ kernel equalities in each; in every case exactly the intended subset
is deficient.  Exact rational arithmetic confirms the same conclusion for
$m\le4$, independently of floating-point rank decisions.

\subsection{The Scott--Vogelius setting}\label{sec:fe}

Let $\Omega=(0,1)^2$.  The hierarchy starts from a $2\times2$ unit-square mesh
with a left diagonal in each square; each level uniformly refines the
triangular macro mesh and applies one barycentric split.  The velocity space
$V_h$ consists of continuous piecewise-quadratic vector fields with homogeneous
Dirichlet data, whose divergence lies in the discontinuous piecewise-linear
Scott--Vogelius pressure space.  This pair is inf-sup stable on barycentrically
refined meshes and gives pointwise divergence-free discrete
velocities~\cite{guzman2018infsup}.  The experiments solve the velocity-block
problem
\begin{equation}
 \begin{aligned}
 a_{h,\tau}(u_h,v_h)
 &=2\bigl(\varepsilon(u_h),\varepsilon(v_h)\bigr)_\Omega\\
 &\quad+\tau\bigl(\nabla\,\cdot u_h,
                     \nabla\,\cdot v_h\bigr)_\Omega,
 \qquad u_h,v_h\in V_h,
 \end{aligned}
  \label{eq:feoperator}
\end{equation}
where $(f,g)_\Omega=\int_\Omega f:g\,dx$ and
$\varepsilon(u)=(\nabla u+\nabla u^\top)/2$.

\paragraph{Relaxation and transfer}
The finite-element operators are assembled with
Firedrake~\cite{rathgeber2017firedrake}, and the additive patches use the
construction of Farrell et al.~\cite{farrell2021pcpatch}.  The preconditioner
applies a W-cycle with an exact coarse solve.  A macro-star is the union of all
macro cells incident to one macro-mesh vertex, and its patch contains the
velocity degrees of freedom supported there~\cite{farrell2021reynolds,
farrell2022elasticity}.  The vertex-star ablation instead uses the cells
incident to one vertex of the barycentrically refined mesh.

On each visit to a noncoarse level, a two-step Chebyshev smoother is
preconditioned by additive macro-star solves with exact patch factorizations.
Its target interval is set by a Krylov estimate of the largest
eigenvalue~\cite{balay2026petsc}.  The intergrid operators use the paired
corrections of Farrell et al.~\cite{farrell2021reynolds}: prolongation removes
a local energy correction from native prolongation and restriction is its
transpose, so the same symmetric patch operator appears in both directions, as
the outer PCG iteration requires.  That correction preserves the discrete
divergence in the large-penalty limit~\cite{farrell2021reynolds} and so
resolves the kernel identified by the one-penalty theory.  The W-cycle is not
the exact additive preconditioner of \cref{eq:schwarz}, so \cref{thm:uniform}
does not apply to it.

\paragraph{Krylov iteration and reported quantities}
PCG starts from zero and stops when the unpreconditioned relative residual
falls below $10^{-7}$, with a limit of 300 iterations.  The primary study uses
a fixed random load with seed 20270829 and penalties $0,1,10^2,10^4,10^6$, and
$10^8$; each value is the median of three solves after one warm-up, with the
true residual recomputed after each.  Writing $T_k$ for the tridiagonal matrix
of the associated Lanczos recurrence, we report
$\widehat{\kappa}=\lambda_{\max}(T_k)/\lambda_{\min}(T_k)$~\cite{greenbaum1997iterative},
a Ritz estimate rather than an exact condition number.  All computations are
serial on an Intel Xeon Cascade Lake processor with Firedrake 2026.4.1 and
PETSc 3.25.0~\cite{rathgeber2017firedrake,balay1997petsc,balay2026petsc}.

\subsection{Testing the criterion under exact additive correction}
\label{sec:fe-exact}

Here the preconditioner has the exact additive form \cref{eq:schwarz}, with
local inverses computed up to roundoff on a fixed finite-element space.  We
evaluate the macro-star kernel conditions, compare them with a deficient
vertex-star construction, and then consider the lattice generated by three
regional penalties.

\subsubsection{Two regional penalties}\label{sec:regional-penalties}

Split the domain into the mesh-aligned regions
$\Omega_L=(0,1/2)\times(0,1)$ and
$\Omega_R=(1/2,1)\times(0,1)$, and weight their divergence terms
independently.  The resulting bilinear form is
\begin{equation}
 \begin{aligned}
 a_{h,\tau_L,\tau_R}(u_h,v_h)
 &=2\bigl(\varepsilon(u_h),\varepsilon(v_h)\bigr)_\Omega
   +\tau_L\bigl(\nabla\,\cdot u_h,\nabla\,\cdot v_h\bigr)_{\Omega_L}\\
 &\quad+\tau_R\bigl(\nabla\,\cdot u_h,
                          \nabla\,\cdot v_h\bigr)_{\Omega_R},
 \end{aligned}
 \label{eq:fe-regional-operator}
\end{equation}
and its matrix has the form
\begin{equation}
 A_h(\tau_L,\tau_R)
 =A_{0,h}+\tau_LK_L^\top K_L+\tau_RK_R^\top K_R.
 \label{eq:fe-regional-matrix}
\end{equation}
Here, $K_L^\top K_L$ and $K_R^\top K_R$ are the matrices of the regional
divergence forms.  The interface follows macro-cell boundaries on every level;
all other solver settings agree with the one-penalty study.

\paragraph{Exact-additive test}
At refinement one the full velocity space has 418 degrees of freedom, of which
354 remain as free algebraic coordinates after imposing the boundary
conditions.  For the left, right, and joint penalties we
compare each global kernel with the span of its intersections with the 25
macro-star correction spaces; the ranks in \cref{tab:fe-regional-audit} use
relative and absolute tolerances $10^{-10}$ and $10^{-12}$.

\begin{table}[tbhp]
\caption{Exact-additive test at refinement one.  Left: dimensions of the
global joint kernels and their assembled local spans.  Right: condition
numbers at six parameter pairs.}
\label{tab:fe-regional-audit}
\centering
\small
\begin{tabular}{lrrr}
\toprule
subset & global dimension & local rank & deficiency\\
\midrule
\MultipenaltySubsetRows
\bottomrule
\end{tabular}
\qquad
\begin{tabular}{rc}
\toprule
$(\tau_L,\tau_R)$ & $\kappa(P^{-1}A_h)$\\
\midrule
\MultipenaltyExactRows
\bottomrule
\end{tabular}
\end{table}

At these tolerances all three kernel dimensions agree with their local span
ranks, the largest local-kernel residual is $5.7\times10^{-16}$, and the six
exact-additive condition numbers range from $13.56$ to $87.45$.  These
floating-point calculations support the kernel conditions on this mesh
but neither prove exact equality nor extend to finer meshes.

\subsubsection{A discretization in which the test fails}
\label{sec:fe-negative}

The macro-star calculation in \cref{tab:fe-regional-audit} has zero deficiency
for every subset.  For comparison, we repeat the audit with the smaller vertex
stars of the barycentrically refined mesh, which fail in the stiff column of
\cref{tab:fe-ablation}.
At refinement one they leave $\ControlFreeDofs$ free coordinates covered by
$\ControlPatches$ patches, and all three subset tests now fail
(\cref{tab:fe-negative}), so \cref{thm:rate} predicts
$\tau\lambda_{\min}\to\gamma_J^{-1}$ along each deficient ray.

\begin{table}[tbhp]
\caption{Vertex-star patches at refinement one.  Left: the three subset tests,
all deficient.  Right: the predicted rate along the joint ray
$\tau_L=\tau_R=\tau$, whose $\gamma_J^{-1}=\ControlJointPredicted$, against the
computed $\tau\lambda_{\min}$ and condition number.}
\label{tab:fe-negative}
\centering
\small
\begin{tabular}{lrrrr}
\toprule
subset & global & local & def. & $\gamma_J$\\
\midrule
\ControlSubsetRows
\bottomrule
\end{tabular}
\qquad
\begin{tabular}{rrrr}
\toprule
$\tau$ & $\tau\lambda_{\min}$ & rel.\ err. & $\kappa$\\
\midrule
\ControlRateRows
\bottomrule
\end{tabular}
\end{table}

At $\tau=10^8$, the measured scaled eigenvalue has relative error
$4.05\times10^{-7}$ with respect to the asymptotic value.  The error decreases
by two orders of magnitude for every two orders of increase in $\tau$, while
the observed condition number grows linearly across six decades.  On the same
mesh, every macro-star deficiency vanishes at the stated tolerances and
$\kappa$ stays below $88$.  Thus the rank diagnosis and the spectral rate are
consistent for these two patch families.

\subsubsection{Three penalties}\label{sec:fe-three}

Two penalties cannot exhibit the non-distributive case distinguished in
\cref{sec:quantitative}, since by \cref{cor:two} their kernel lattice is always
distributive.  We therefore cut $\Omega$ into
three equal vertical strips, whose interfaces lie on cell boundaries when the
base resolution is even and divisible by three; $\ThreeBaseN\times\ThreeBaseN$
gives $\ThreeFreeDofs$ free coordinates and $\ThreePatches$ macro-star patches
at the coarsest level.

\begin{table}[tbhp]
\caption{Three regional penalties with macro-star patches.  Left: all seven
subset tests.  Right: exact additive condition numbers over the orthant sample
points.  The partition is symmetric under $x\mapsto1-x$, so $\kappa$ is
unchanged by reversing the weights and one member of each mirror pair is
omitted.}
\label{tab:fe-three}
\centering
\small
\begin{tabular}{lrrrr}
\toprule
$J$ & global & local & def. & $\gamma_J$\\
\midrule
\ThreeSubsetRows
\bottomrule
\end{tabular}
\qquad
\begin{tabular}{lr}
\toprule
$(\tau_1,\tau_2,\tau_3)$ & $\kappa(P^{-1}A_h)$\\
\midrule
\ThreeOrthantRows
\bottomrule
\end{tabular}
\end{table}

The macro-star patches are deficient by one dimension for the middle strip and
have zero deficiency for the other six subsets at the stated tolerances.  The
middle strip is the only region with two interior interfaces.  The rate
calculation provides an independent consistency check on the numerical rank
diagnosis: $\gamma_{\{2\}}=\ThreeMiddleGamma$ gives
$\tau\lambda_{\min}\to\ThreeMiddlePredicted$, and at $\tau=10^8$ the computed
value has relative error $\ThreeMiddleTailError$.  Thus the three-region
partition introduces a subset-specific obstruction absent from the two-region
test: only the middle singleton kernel is deficient; the outer singletons, all
pairs, and the full intersection pass.

The condition numbers also display the subset-specific behavior of
\cref{ex:counterexample}.  They stay below $500$ on every sampled ray except
those on which $\tau_2$ diverges alone or dominates, corresponding to the
deficient singleton subset.  Inspecting only the equal-weight path would miss
this deterioration: the equal-weight sample gives $\kappa=420.58$, whereas
the value at $(1,10^8,1)$ is $2.08\times10^8$.

The kernel lattice is informative independently of this patch deficiency,
because it depends only on the joint kernels.  The three singleton projectors
do not commute, the largest pairwise commutator norm being
$\ThreeCommutator$, so \cref{cor:commuting} does not apply; closing the lattice
generated by the seven joint kernels gives $\ThreeLatticeMembers$ members, on
which the distributive law fails.  An $M_3$-type obstruction
(\cref{rem:m3}) therefore occurs in a Scott--Vogelius discretization and not
only in the $2\times2$ example.  Thus, even after enriching the patches until
all subset tests pass, \cref{thm:distributive} would not certify a common
splitting, although the ordering-cone construction would remain available.
Distributivity is sufficient and not necessary, so its failure does not prove
that no common splitting exists; moreover, the verdict is a floating-point
decision at the stated tolerances rather than a proof.

\subsection{Testing the mechanism in an inexact W-cycle}
\label{sec:fe-cycle}

The W-cycle is neither exact nor purely additive, so the exact-additive results
\cref{thm:uniform,thm:characterization} do not apply.  We therefore examine
whether its observed parameter dependence is consistent with a
kernel-preserving design for relaxation and transfer.

\Cref{tab:fe-primary} reports all 30 mesh--parameter combinations of the
primary study; at refinement $r$ there are $96\,4^{r-1}$ cells and $r+1$
levels.  All 90 solves converge, with largest recomputed relative residual
$8.383\times10^{-8}$.

\begin{table}[tbhp]
\caption{Scott--Vogelius penalty sweep.  Each entry gives the PCG iteration
count followed by the Ritz condition estimate $\widehat{\kappa}$, separated
by a slash.  Both values are medians of three repeated solves.}
\label{tab:fe-primary}
\centering
\scriptsize
\setlength{\tabcolsep}{3.3pt}
\begin{tabular}{rrcccccc}
\toprule
& & \multicolumn{6}{c}{$\tau$}\\
\cmidrule(lr){3-8}
refinement & velocity DOFs & $0$ & $1$ & $10^2$ & $10^4$ & $10^6$ & $10^8$\\
\midrule
1 & 418    & $7\,/\,1.384$ & $7\,/\,1.341$ & $10\,/\,2.368$ & $10\,/\,2.562$ & $10\,/\,2.565$ & $10\,/\,2.565$\\
2 & 1,602  & $7\,/\,1.378$ & $7\,/\,1.340$ & $12\,/\,2.575$ & $12\,/\,2.721$ & $12\,/\,2.722$ & $12\,/\,2.722$\\
3 & 6,274  & $7\,/\,1.388$ & $7\,/\,1.347$ & $12\,/\,2.671$ & $13\,/\,2.813$ & $13\,/\,2.813$ & $13\,/\,2.813$\\
4 & 24,834 & $7\,/\,1.391$ & $7\,/\,1.350$ & $12\,/\,2.651$ & $13\,/\,2.792$ & $13\,/\,2.793$ & $13\,/\,2.793$\\
5 & 98,818 & $7\,/\,1.390$ & $7\,/\,1.350$ & $12\,/\,2.662$ & $13\,/\,2.806$ & $13\,/\,2.807$ & $13\,/\,2.807$\\
\bottomrule
\end{tabular}
\end{table}

Raising the penalty from $1$ to $10^2$ adds three to five iterations, whereas
six further orders of magnitude add at most one more: at $\tau=10^8$ the count
is 10--13 with $\widehat{\kappa}$ between 2.565 and 2.813.  The last three
columns are nearly equal on every mesh, indicating saturation over the sampled
range.

\Cref{tab:fe-ablation} compares the kernel-preserving configuration with five
alternatives at refinement three, including
BoomerAMG~\cite{henson2002boomeramg}.

\begin{table}[tbhp]
\caption{Finite-element ablations at refinement three; entries are PCG
iteration counts.  An asterisk marks failure to meet the tolerance within 300
iterations, and a dagger marks a PETSc breakdown after the displayed
iteration.}
\label{tab:fe-ablation}
\centering
\small
\begin{tabular}{llrrr}
\toprule
relaxation or preconditioner & transfer & $\tau=0$ & $\tau=1$ & $\tau=10^8$\\
\midrule
macro-star patches & corrected & 7 & 7 & 13\\
macro-star patches & native & 7 & 7 & $300^\ast$\\
vertex-star patches & corrected & 10 & 11 & $300^\ast$\\
vertex-star patches & native & 10 & 12 & $38^\dagger$\\
Jacobi multigrid & native & 17 & 20 & $39^\dagger$\\
BoomerAMG & algebraic & 9 & 11 & $300^\ast$\\
\bottomrule
\end{tabular}
\end{table}

At $\tau=0$ and $\tau=1$, every method converges, and the table does not
distinguish the geometric choices.  At $\tau=10^8$, only the macro-star method
with corrected transfer reaches the tolerance.  Replacing either the
macro-star patches or the corrected transfer leads to failure or breakdown.
BoomerAMG also reaches the iteration limit; because its symmetry was not
established, this PCG result is reported only as a solver control and is not
interpreted as spectral evidence.

\paragraph{Two regional weights}
The component-matched transfer uses the piecewise coefficient
$\tau_L\mathbf 1_{\Omega_L}+\tau_R\mathbf 1_{\Omega_R}$ in both the transfer
energy and its penalty-only right-hand side.  We test six parameter pairs on
each of five refinements, reporting medians of three solves after one warm-up,
and complete the grid $\{1,10^2,10^4,10^6,10^8\}^2$ with nineteen further pairs
at refinement three.  Single-solve controls there compare component-matched
transfer with native transfer and with corrected transfer at the uniform
coefficient $\max(\tau_L,\tau_R)$.

\begin{table}[tbhp]
\caption{Median PCG iterations for the component-matched W-cycle at the six
primary parameter pairs; each entry summarizes three repeated solves after
one warm-up solve.}
\label{tab:fe-regional-primary}
\centering
\scriptsize
\setlength{\tabcolsep}{3.3pt}
\begin{tabular}{rrcccccc}
\toprule
& & \multicolumn{6}{c}{$(\tau_L,\tau_R)$}\\
\cmidrule(lr){3-8}
refinement & velocity DOFs
 & $(1,1)$ & $(10^8,1)$ & $(1,10^8)$ & $(10^8,10^8)$
 & $(10^8,10^2)$ & $(10^2,10^8)$\\
\midrule
\MultipenaltyPrimaryRows
\bottomrule
\end{tabular}
\end{table}

\begin{table}[tbhp]
\caption{Extreme single-solve transfer controls at refinement three.  An
asterisk marks failure at the 300-iteration limit.}
\label{tab:fe-regional-grid}
\centering
\small
\begin{tabular}{lccc}
\toprule
transfer & $(10^8,1)$ & $(1,10^8)$ & $(10^8,10^8)$\\
\midrule
\MultipenaltyControlRows
\bottomrule
\end{tabular}
\end{table}

All $(6\cdot5+19)\cdot3=147$ component-matched solves converge in 7--13
iterations, including the six primary pairs at refinement five with 98,818
velocity DOFs.  Every
entry in the complete 25-point refinement-three grid lies between 7 and 13,
with the value 7 occurring only at $(\tau_L,\tau_R)=(1,1)$.  At
refinement three, native transfer fails to reach the tolerance within 300
iterations in every extreme control (\cref{tab:fe-regional-grid}), whereas
uniform corrected transfer converges in 13--14.  The controls thus separate
corrected from native
transfer, but do not establish that component matching is necessary, the
uniform correction having been tested at only the three displayed pairs.

The regional divergence norms reflect the imposed weights.  At refinement
three, the left/right pair changes from
$(\MultipenaltyBaselineLeftDivergence,\MultipenaltyBaselineRightDivergence)$
at $(\tau_L,\tau_R)=(1,1)$ to
$(\MultipenaltyLeftStiffLeftDivergence,
\MultipenaltyLeftStiffRightDivergence)$ at $(10^8,1)$ and
$(\MultipenaltyRightStiffLeftDivergence,
\MultipenaltyRightStiffRightDivergence)$ at $(1,10^8)$.

\section{Conclusion}\label{sec:conclusion}

For a fixed finite-dimensional exact-additive preconditioner, robustness over
the nonnegative parameter orthant is equivalent to decomposability of the
joint kernel associated with every nonempty subset of penalties.  A deficient
subset produces linear-order condition-number growth along its active ray, and
the generalized eigenvalue $\gamma_J$ gives the exact leading coefficient for
the stable-decomposition constant and the reciprocal smallest eigenvalue.  The
subset conditions are irredundant in the worst case.  When they hold, filtered
right inverses give computable bounds on parameter-ordering cones; a
distributive kernel lattice, which occurs automatically for two penalties,
permits one common splitting over the full orthant.

The numerical study separates finite-dimensional evaluation of these results
from evidence for an inexact multilevel solver.  In the staggered-grid and
exact-additive finite-element calculations, computed ranks and spectra are
consistent with the subset characterization and the $\gamma_J$ asymptotics.
For the three-region Scott--Vogelius problem, the middle singleton subset is
deficient at the stated tolerances, and the unequal-weight samples deteriorate
although the equal-weight sample remains moderate.  The separate W-cycle
experiments converge in 7--13 PCG iterations over the sampled mesh and parameter
ranges when macro-star relaxation and corrected transfer are used.

\paragraph{Limitations and future work}
The filtered-construction constants are finite on each mesh, but their uniformity under refinement is left to future work.  The W-cycle is inexact, and extending the analysis to inexact local solves and multilevel composition
remains an open problem.  Further directions include enriching the relaxation for the three-region case, treating general regional partitions, and analyzing
the corresponding complete saddle-point systems.

\section*{Acknowledgments}
The author thanks the School of Computing at The Australian National University for providing computational resources. 
AI tools are used for language and grammar checks only; they do not contribute to the mathematical content of this work.

\end{document}